\documentclass[10pt]{article}

\usepackage{graphicx}

\usepackage{amsmath}

\usepackage{amssymb}

\usepackage{bm}

\usepackage{amscd}

\usepackage[top=13truemm
, bottom=20truemm, left=30truemm, right=30truemm]{geometry}

\usepackage{setspace}
\usepackage{amsthm}

\usepackage{array}

\newcolumntype{I}{!{\vrule width 2pt}}

\newtheorem{dfn}{Definition}[subsection]
\newtheorem{pp}[dfn]{Proposition}
\newtheorem{col}[dfn]{Corollary}
\newtheorem{lem}[dfn]{Lemma}

\newtheorem{thrm}[dfn]{Theorem}
\newtheorem{fact}[dfn]{Fact}

\def\M{\operatorname{Mat}}

\def\A{\operatorname{Alt}}
\def\S{\operatorname{Sym}}
\def\G{\operatorname{GL}}

\def\tr{\operatorname{trace}}
\def\rk{\operatorname{rank}}
\def\wt{\operatorname{wt}}

\def\ker{\operatorname{ker}}
\def\sgn{\operatorname{sgn}}
\def\diag{\operatorname{diag}}
\def\even{\operatorname{even}}
\def\odd{\operatorname{odd}}

\def\Aff{\operatorname{Aff}}

\newcommand{\1}{\Psi^{(1)}}
\newcommand{\2}{\Psi^{(2)}}
\newcommand{\3}{\Psi^{(3)}}
\newcommand{\4}{\Psi^{(4)}}
\newcommand{\calo}{\mathcal{O}}
\newcommand{\calp}{\mathcal{P}}
\newcommand{\calf}{\mathcal{F}}
\newcommand{\bbf}{\mathbb{F}}

\title{\bf $q$-
Hypergeometric Polynomials and Group-invariant Fourier Transformations over a Finite Field}
\author{Koei KAWAMURA}
\date{}

\begin{document}

\maketitle

\begin{abstract}
By the Fourier transformations, any group-invariant functions over finite Abelian groups are transformed into group-invariant functions over the character groups.
In this paper, we calculate matrix elements of this transformations under specific bases.
More specifically, we deal with some vector spaces over a finite field and linear actions.
Then the matrix elements under adequate bases are expressed by Krawtchouk or Affine $q$-Krawtchouk polynomials.
For calculations, we construct a commutative diagram which combines two settings of group-invariant Fourier transformations.
We apply it to different sized transformations of each example, and solve it inductively.
We remark the matrix elements are related to the zonal spherical functions of finite Gelfand pairs.
\end{abstract}

\section{Introduction}\label{intro}

As our general setting, let a finite group $G$ act on a finite Abelian group $A$ as group isomorphisms.
Then $G$ also acts on the character group $\hat{A}$ of $A$ by the contragredient action.
In this paper, we study Fourier transformations of $G$-invariant functons on $A$ and $\hat{A}$.
Let $\mathbb{C}[A]$ and $\mathbb{C}[\hat{A}]$ be the spaces of complex valued functions on $A$ and $\hat{A}$ respectively.
For any function $\varphi \in \mathbb{C}[A]$, we define the Fourier transformation $\mathcal{F} \varphi \in \mathbb{C}[\hat{A}]$ by
\begin{equation}
 \label{intro-f-henkan}
\mathcal{F}\varphi(\xi)\;= \sum_{a \in A} \overline{\xi(a)} \varphi(a)  \qquad(\xi \in \hat{A}).
\end{equation}
Also the inverse transformation $\bar{\calf}$ is defined, and
the map $\mathcal{F}:\mathbb{C}[A] \rightarrow \mathbb{C}[\hat{A}]$ is a linear isomorphism.
The $G$-actions on $A$ and $\hat{A}$ are brought up to the spaces $\mathbb{C}[A]$ and $\mathbb{C}[\hat{A}]$ respectively, and they are commutative with the transformation $\mathcal{F}$.
Thus we can consider the `group-invariant' Fourier transformation $\mathcal{F}:\mathbb{C}[A]^G \rightarrow \mathbb{C}[\hat{A}]^G$ by restricting $\mathcal{F}$ to the spaces of $G$-invariant functions.
By regarding $G$-invariant functions as functions on the orbits set $_G\! \backslash A$ or $_G\! \backslash A$,
we may write
\begin{equation}
 \label{intro-g-f}
\mathcal{F}:\mathbb{C}[_G\! \backslash A] \rightarrow \mathbb{C}[_G\! \backslash \hat{A}].
\end{equation}
We can take `canonical' bases of the spaces $\mathbb{C}[_G\! \backslash A]$ and $\mathbb{C}[_G\! \backslash \hat{A}]$
that consist of characteristic functions $\chi_{\calo}$ of orbits $\calo \in \,_G\! \backslash A$ or $\chi_{\calp}$ of orbits $\calp \in\,_G\! \backslash \hat{A}$.
We define `canonical' matrix $\varPhi= \left( \varPhi(\mathcal{P}, \mathcal{O})\right)_{\calp \in\, _G\! \backslash \hat{A}, \calo \in\, _G\! \backslash A}$ of the transformation (\ref{intro-g-f}) under these bases. 
By definition,
\begin{equation}
 \label{intro-me}
\varPhi(\mathcal{P}, \mathcal{O})=\; \mathcal{F}\chi_{\mathcal{O}}(\mathcal{P}) =\; \sum_{a \in \mathcal{O}} \overline{\xi(a)}, \quad \text{where}\;\;\xi \in \mathcal{P}.
\end{equation}
We calculate it for some pairs $(A,G)$.
Specifically in any cases, $A$ is a vector space over a finite field $\bbf$, and the $G$-action is $\bbf$-linear and has a good property (called adjoint-free in Definition \ref{adjf}).
Then we can correspond the orbits sets $_G \backslash A$ and $_G \backslash \hat{A}$ and parametrize both by a set $\Lambda$.
So $\varPhi$ is defined on $\Lambda \times \Lambda$.

The canonical matrix elements in our examples are expressed by Krawtchouk polynomials or Affine $q$-Krawtchouk polynomials.
These special functions are defined in terms of the terminating Gaussian hypergeometric function $_2F_1$ or the basic hypergeometric function $_3\varphi_2$ as follows \cite{Koe}, \cite{S}:
\begin{dfn}
 \label{kraw}
\; For non-negative integers $y$ and $N$ such that $y \leq N$ and parameter $p$, the Krawtchouk polynomial is defined by
$$
K_y(x;p,N)=\,
{}_2F_1
\begin{pmatrix}
-y,\; -x   \\[-12pt]
 & ; & \!\! \dfrac{1}{p} \; \\[-12pt]
 \; -N &  
\end{pmatrix}
= \; \sum_{k=0}^y\, \frac{(-y)_k \; (-x)_k}{(-N)_k \; k! \; p^k}.
$$
This is a polynomial in $x$ of degree $y$.
Here $(a)_k=a \, (a+1) \cdot \dots \cdot (a+k-1), (a)_0=1$.
\end{dfn}

\begin{dfn}
 \label{affqkraw}
\; For non-negative integers $y$ and $N$ such that $y \leq N$ and parameter $a, q$, the Affine $q$-Krawtchouk polynomial is defined by
$$
K^{\Aff}_y(x;a,N;q)=\;
{}_3\varphi_2
\begin{pmatrix}
q^{-y}, \; q^{-x},\; 0   \\[-6pt]
 & ; & \!\! q, & q \; \\[-6pt]
\; a, \;q^{-N} &  
\end{pmatrix}
=\;
\sum_{k=0}^{y}\, \frac{(q^{-y};q)_k \; (q^{-x};q)_k}{(q^{-N};q)_k \; (a;q)_k \; (q;q)_k}\, q^k.
$$
This is a polynomial in $q^{-x}$ of degree $y$.
Here $(a;q)_k=(1-a) \, (1-aq) \, (1-aq^2) \cdot \dots \cdot \, (1-aq^{k-1}), (a;q)_0=1$.
\end{dfn}
\noindent
We remark Affine $q$-Krawtchouk polynomial is one of $q$-analogues of Krawtchouk polynomial \cite{Koe}.
We use also the Gaussian-polynomial
defined by
\begin{equation}
 \label{gau}
\begin{bmatrix} N \\ x \end{bmatrix} _q =\; \frac{(q^N;q^{-1})_x}{(q;q)_x} \qquad (x \leq N),
\end{equation}
which is a $q$-analogue of the binomial coefficient $\binom{N}{x}$.

The first example we deal with is the case $(A, G)=(\mathbb{F}^n, (\mathbb{F}^\times)^n \rtimes \mathfrak{S}_n)$, where $\mathbb{F}^\times$ is the multiplicative group of $\mathbb{F}$ and $\mathfrak{S}_n$ is the symmetric group.
Then orbits in $\bbf^n$ are characterized by weights of the elements, and parametrized by the set $\{0,1, \cdots, n\}$.
The canonical matrix elements are as follows:
\begin{pp}
\label{intro-diag-me}
\;For $s,r \in \{0,1, \cdots, n\}$,
\begin{equation}\nonumber
\varPhi_n(s,r)=\;(q-1)^r\, \binom{n}{r}\, K_s(r;\frac{q-1}{q},n),
\end{equation}
\end{pp}
\noindent
where $q$ is the order of $\bbf$.
In the other examples, we work with additive groups of matrices over $\mathbb{F}$.
Let $\M_{n,m}(\mathbb{F})$ be $n \times m$-sized all matrices, where $n \leq m$ for simplicity.
The group $\G_n(\bbf) \times \G_m(\bbf)$ acts on $\M_{n,m}(\bbf)$ by multiplication, and the orbits are parametrized by rank $r \in \{0,1, \cdots, n\}$ of the elements.
The canonical matrix elements of the case $(A,G)=\left(\M_{n,m}(\bbf), \G_n(\bbf) \times \G_m(\bbf)\right)$ is as follows:
\begin{pp}
 \label{intro-mat-me}
\;For $s,r \in \{0,1, \cdots, n\}$,
\begin{equation}\nonumber
\varPhi_{n,m}(s,r)
=\; (-1)^rq^{\binom{r}{2}} (q^m;q^{-1})_r \begin{bmatrix} n \\ r \end{bmatrix} _q \, K^{\Aff}_s(r; q^{-m}, n; q).
\end{equation}
\end{pp}
\noindent
The next example is the case $A=\A_{n}(\bbf)$ of $n \times n$-sized alternating matrices with an action of $G=\G_n(\bbf)$.
Then also the canonical matrix elements are expressed by Affine $q$-Krawtchouk polynomials (see Prop.\ref{alt-me}).
The last example is the case $A=\S_n$ of $n \times n$-sized symmetric matrices with actions of $G=\G_n$ or $G=\bbf^\times \times \G_n$.
In this example only, we calculate matrix elements under adequately changed bases from the canonical bases.
Then again they are expressed by Affine $q$-Krawtchouk polynomials.
See sec.\ref{5-1} for the base change, and see Prop.\ref{psi1}, Prop.\ref{psi2}, Prop.\ref{psi3} and Prop.\ref{psi4} in sec.\ref{5-4} for the values of the matrix elements.

Calculations in the examples above are all based on a commutative diagram which combines two settings of group-invariant Fourier transformations.
Let $(A, G)$ and $(B, H)$ be pairs of finite Abelian groups and acting groups, $\pi: A \to B$ be a group homomorphism, and $\zeta \in \hat{A}$ be any character.
Define $\hat{\pi}_{\zeta}: \hat{B} \to \hat{A}$ by
\begin{equation}
 \label{intro-hatpi}
\hat{\pi}_{\zeta}\xi =\; (\xi \circ \pi)\cdot \zeta \qquad (\xi \in \hat{B}),
\end{equation}
and assume it preserves $G$- and $H$-orbits 
 (or some equivalent conditions ($\natural$) in sec.\ref{2-1}).
Then we can define transformations $\pi^{\zeta}_{*}: \mathbb{C}[_G\backslash A] \rightarrow \mathbb{C}[_H\backslash B]$ and
$\hat{\pi}_{\zeta}^*: \mathbb{C}[_G\backslash \hat{A}] \rightarrow \mathbb{C}[_H\backslash\hat{B}]$ by
\begin{equation}
 \label{intro-osidasi}
\pi^{\zeta}_{*}\varphi(b)=\; \sum_{a \in \pi^{-1}(b)} \overline{\zeta(a)} \varphi(a)\qquad (\varphi \in \mathbb{C}[_G\backslash\!A],\; b \in B),
\end{equation}
and
$
\hat{\pi}_{\zeta}^* \psi=\psi \circ \hat{\pi}_\zeta\;( \psi \in \mathbb{C}[_G\backslash\!\hat{A}]),
$
and we have a commutative diagram as follows:
\begin{equation}
 \label{intro-d2}
\begin{CD}
\mathbb{C}[_G\!\backslash A] @>{\mathcal{F}_A}>>\mathbb{C}[_G\!\backslash \hat{A}] \\
@V{\pi^{\zeta}_{*}}V{\qquad\quad\circlearrowright}V  @VV{\hat{\pi}_{\zeta}^*}V \\
\mathbb{C}[_H\!\backslash B] @>>{\mathcal{F}_B}> \mathbb{C}[_H\!\backslash \hat{B}]
\end{CD}
\end{equation}
In the specific considerations, we apply the diagram (\ref{intro-d2}) to different sized spaces, like $\bbf^n$ and $\bbf^{n-1}$ as $A$ and $B$.
Then we can derive recursions of canonical matrix elements (or matrix elements under changed bases) and solve them.
As a remark, the diagram (\ref{intro-d2}) gives not only the relation $\hat{\pi}_\zeta^* \circ \calf_A=\calf_B \circ \pi_*^\zeta$ but also the relation $\pi^\zeta_* \circ \bar{\calf}_A=\bar{\calf}_B \circ \hat{\pi}^*_\zeta$ concerning inverse transformations, and each relation gives independent recursion from the other.
In order to solve the matrix of $\pi_*^\zeta$, we need some calculations in each example.
We can also induce some properties of the matrix elements in our considerations.
In the examples of $\bbf^n$ and $\M_{n,m}(\bbf)$,
we derive orthogonality relations and generating functions of Krawtchouk and Affine $q$-Krawtchouk polynomials respectively (Although they are already known \cite{Koe}).


We also investigate concrete forms of the inverse transformation $\bar{\calf}$.  
If $A$ is a vector space over $\mathbb{F}$ and the action of $G$ has a good property, then we see $\calf$ and $\bar{\calf}$ have almost the same forms.
Thus for a function $\varphi \in \mathbb{C}[_G\!\backslash A]$, let $\hat{\varphi}$ be $\mathcal{F}\varphi$ with a small adjustment, then $\hat{\hat{\varphi}}=\varphi$ holds.
Furthermore in some examples, we can get $q$-analogues of one fundamental inversion formula as follows:
\begin{pp}
 \label{mori}
\; For a function $\varphi:\{0,1, \dots,n\}\longrightarrow\mathbb{C}$, if we define
\begin{equation}
\label{morimori}
\hat{\varphi}(s)=\sum_{r=0}^{s}(-1)^{r}\binom{s}{r} \varphi(r) \qquad(s=0,1,\dots,n),
\end{equation}
then $\hat{\hat{\varphi}}=\varphi$ holds.
\end{pp}
\noindent
Examples are given in Corollaries \ref{seisitu1}, \ref{seisitu2} and \ref{seisitu3}.

We should remark that the canonical matrix elements are related to the zonal spherical functions of a finite Gelfand pair $(A \rtimes G, G)$.
For each orbit $\mathcal{P} \in\, _G\! \backslash \hat{A}$, the spherical representation $W_{\mathcal{P}} \subset \mathbb{C}[A]$ is constructed.
The zonal spherical function $\omega_{\mathcal{P}}$ corresponding to $W_{\mathcal{P}}$ is
the unique $G$-invariant function in $W_{\mathcal{P}}$ satisfying $\omega_{\mathcal{P}}(0)=1$.
Then the relation
\begin{equation}
 \label{intro-kyu}
\omega_{\mathcal{P}}(\mathcal{O})
=\frac{1}{|\mathcal{O}|} \overline{ \varPhi(\mathcal{P}, \mathcal{O})}
\end{equation}
holds.

%

In preceeding reserches, 
our matrix elements are studied as spherical functions or eigenvalues of association shemes.
Dunkl \cite{Dun} and Koornwinder \cite{K} gave the interpretations of Krawtchouk polynomials as the spherical functions on the wreath product of symmetric groups.
This result is the same as the canonical matrix elements of our example of $\bbf^n$, since the actions can be regrded as equivalent.
According to Delsarte \cite{D} and Delsarte and Goethals \cite{DG}, Affine $q$-Krawtchouk polynomials are the eigenvalues of association shemes of bilinear and alternating bilinear forms over $\mathbb{F}$.
And Stanton \cite{S} pointed out their results can be seen as spherical functions on
$\M_{n,m}(\bbf)$ and $\A_{n}(\bbf)$ respectively.
As for the example of $\S_n(\bbf)$, Egawa \cite{Ega} constructed an association scheme about quadratic forms.
The eigenvalues of it are similar to our matrix elements.
Thus our canonical matrix elements had been culculated in foregoing reserches.
But in this paper, we may give new interpretations and
another derivations which are based on the group-invariant Fourier transformations on a finite field.


\vspace{0.1in}\noindent
\underline{NOTATIONS.}\; For a set $S$, let both $|S|$ and $\sharp S$ denote the cardinality of $S$.
For any $x$ and $y$, let $\delta_{xy}$ denote Kronecker's delta, that is,
$$
\delta_{xy}=\; \begin{cases}
1 & (x=y) \\
0 & (x \neq y).
\end{cases}
$$


\newpage
%

\section{Fourier Transformation and zonal spherical functions}\label{one}

Throughout this section, let $A$ be a finite Abelian group and $G$ be a finite group which has an action $\rho$ on $A$ as group isomorphisms.
In this section, we define Fourier transformation of $G$-invariant functions on $A$ in general.
And we investigate fundamental properties of the transformation.
Especially we describe relations with the zonal spherical functions of a finite Gelfand pair $(A \rtimes G, G)$.

\subsection{Group-invariant Fourier transformation on a finite Abelian group}

Let $\hat{A}$ be the character group of $A$,
$\mathbb{C}[A]$ and $\mathbb{C}[\hat{A}]$ be $\mathbb{C}$-vector spaces of complex valued functions on $A$ and $\hat{A}$ respectively.
In general, the function $\varphi \in \mathbb{C}[A]$ is transformed into $\mathcal{F} \varphi \in \mathbb{C}[\hat{A}]$ by {\bf Fourier transformation} as follows:
\begin{equation}
 \label{f-henkan}
\mathcal{F}\varphi(\xi)\;=\; \sum_{a \in A} \overline{\xi(a)} \varphi(a) \qquad(\xi \in \hat{A}),
\end{equation}
where the overline means complex conjugate.
The inverse transformation is given as follows: For $\psi \in \mathbb{C}[\hat{A}]$,
\begin{equation}
 \label{gyaku-henkan}
\bar{\mathcal{F}}\psi(a) \;=\; \frac{1}{|A|} \sum_{\xi \in \hat{A}} \xi(a) \psi(\xi) \qquad (a \in A).
\end{equation}
The reversible relations $\bar{\mathcal{F}} \mathcal{F} \varphi = \varphi$ and $\mathcal{F} \bar{\mathcal{F}} \psi = \psi$ are easily verified.

Now let us consider $G$-action on $A$.
Then it induces the contragredient action $\hat{\rho}$ on $\hat{A}$ by $\hat{\rho}(g)\xi = \xi \circ \rho(g^{-1}),\; \xi \in \hat{A},\; g \in G$.
These actions are brought up to $G$-representations $(\rho, \mathbb{C}[A])$ and $( \hat{\rho}, \mathbb{C}[\hat{A}])$  respectively (we use the same letters $\rho$ and $\hat{\rho}$).
We can verify that the transformation $\mathcal{F}$ is an intertwining operator of these representations, 
that is, $\mathcal{F} \circ \rho(g) =\, \hat{\rho}(g) \circ \mathcal{F}$ holds for $\forall g \in G$.
Therefore $\mathcal{F}$ preserves $G$-invariance of functions, and so does $\bar{\mathcal{F}}$.
Let us write $_G\! \backslash A$ for all $G$-orbits of $A$ and regard $\mathbb{C}[_G\! \backslash A]$ as a subspace of all $G$-invariant functions of $\mathbb{C}[A]$.
Likewise $_G\! \backslash \hat{A}$ and $\mathbb{C}[_G\! \backslash \hat{A}]$ are.
Then we can consider the group-invariant Fourier transformations
\begin{equation}
 \label{g-f}
\mathcal{F}:\mathbb{C}[_G\! \backslash A] \rightarrow \mathbb{C}[_G\! \backslash \hat{A}],
\end{equation}
\begin{equation}
 \label{g-barf}
\bar{\mathcal{F}}:\mathbb{C}[_G\! \backslash \hat{A}] \rightarrow \mathbb{C}[_G\! \backslash A],
\end{equation}
by the restriction of $\calf$ and $\bar{\calf}$.
In this paper we investigate properties of these transformations.

\subsection{Canonical matrix}\label{1-2}

For any set $S \subset A$, let $\chi_{S}$ be its characteristic function on $A$. 
Then $\mathbb{C}[_G\! \backslash A]$ has a $\mathbb{C}$-basis
$
\{ \chi_{\mathcal{O}} \mid \mathcal{O} \in\, _G\!\backslash A \}$.
Likewise 
$
\{ \chi_{\mathcal{P}} \mid \mathcal{P} \in\, _G\!\backslash \hat{A} \}$ consists a basis of $\mathbb{C}[_G\! \backslash \hat{A}]$.
We call them canonical bases.
Let $\varPhi= \left( \varPhi(\mathcal{P}, \mathcal{O})\right)$ be the matrix of the transformation $\mathcal{F}$ in (\ref{g-f}) on these bases.
We call it {\bf canonical matrix} of the pair $(A, G)$.
By definition,
\begin{equation}
 \label{me}
\varPhi(\mathcal{P}, \mathcal{O})=\; \mathcal{F}\chi_{\mathcal{O}}(\mathcal{P}) =\; \sum_{a \in \mathcal{O}} \overline{\xi(a)} \qquad (\mathcal{O} \in\, _G\! \backslash A,\; \mathcal{P} \in\, _G\! \backslash \hat{A},\; \xi \in \mathcal{P}).
\end{equation}
On the other hand let $\bar{\varPhi}=\left( \bar{\varPhi}(\mathcal{O}, \mathcal{P}) \right)$ be the matrix of $|A|\bar{\mathcal{F}}$ in (\ref{g-barf}) on the same bases.
Then by definition,
\begin{equation}
 \label{barme}
\bar{\varPhi}(\mathcal{O}, \mathcal{P})
=\; |A|\bar{\mathcal{F}} \chi_{\mathcal{P}}(\mathcal{O})
=\; \sum_{\xi \in \mathcal{P}} \xi(a)
\qquad
(\mathcal{O} \in\, _G\! \backslash A,\; \mathcal{P} \in\, _G\! \backslash \hat{A},\; a \in \mathcal{O}).
\end{equation}
For any $G$-invariant functions $\varphi \in \mathbb{C}[_G\!\backslash A]$ and $\psi \in \mathbb{C}[_G\!\backslash \hat{A}]$, Fourier transformations are given as follows:
\begin{eqnarray}
 \label{ghuhen-henkan}
\mathcal{F}\varphi(\mathcal{P})=
\; \sum_{\mathcal{O} \in \, _G\! \backslash A} \varPhi(\mathcal{P}, \mathcal{O}) \varphi(\mathcal{O}) \qquad (\mathcal{P} \in \,_G\!\backslash \hat{A}),\\
\label{ghuhen-gyaku-henkan}
\bar{\mathcal{F}}\psi(\mathcal{O})=
\; \frac{1}{|A|}\sum_{\mathcal{P} \in \, _G\! \backslash \hat{A}} \bar{\varPhi}(\mathcal{O}, \mathcal{P}) \psi(\mathcal{P}) \qquad (\mathcal{O} \in \,_G\!\backslash A).
\end{eqnarray}
By (\ref{me}) and (\ref{barme}), we can verify
\begin{equation}
 \label{me-barme}
\bar{\varPhi}(\mathcal{O},\mathcal{P})=\; \frac{|\mathcal{P}|}{|\mathcal{O}| }\overline{\varPhi(\mathcal{P},\mathcal{O})}
\qquad
(\mathcal{O} \in\, _G\! \backslash A,\; \mathcal{P} \in\, _G\! \backslash \hat{A}).
\end{equation}
The reversible equation $\mathcal{F} \bar{\mathcal{F}}=$id and equation (\ref{me-barme}) induces an orthogonality relation of canonical matrix elements $\varPhi(\mathcal{P}, \mathcal{O})$ as follows:
\begin{equation}
 \label{me-cho}
\sum_{\mathcal{P} \in \,_G\! \backslash \hat{A}} |\mathcal{P}|\, \varPhi(\mathcal{P},\mathcal{O}) \overline{\varPhi(\mathcal{P},\mathcal{O}')}
=\; \delta_{\mathcal{O} \mathcal{O}'} |\mathcal{O}||A|
\qquad (\mathcal{O},\; \mathcal{O}' \in\,_G\! \backslash A).
\end{equation}
But more generally, we can derive a kind of `multi-orthogonality' relation as follows:
\begin{pp}
 \label{taju}
Let $k \geq 1$ and $\mathcal{O}_1, \dots, \mathcal{O}_k, \mathcal{O} \in \,_G\!\backslash A$.
Then we have
$$
\sum_{\mathcal{P} \in \,_G\! \backslash \hat{A}}|\mathcal{P}|\, \varPhi(\mathcal{P}, \mathcal{O}_1)\cdots \varPhi(\mathcal{P}, \mathcal{O}_k)\overline{\varPhi(\mathcal{P}, \mathcal{O})}
=\; \sharp \{(a_1,\cdots, a_k) \in \mathcal{O}_1 \times \cdots \times \mathcal{O}_k \mid a_1+ \cdots +a_k \in \mathcal{O}\}\;|A|.
$$
\end{pp}
\noindent
\underline{Remark}.\; When $k=1$, since $\sharp \{(a_1) \in \mathcal{O}_1 \mid a_1 \in \mathcal{O}\}=\delta_{\mathcal{O}_1 \mathcal{O}}|\mathcal{O}|$, it becames back to (\ref{me-cho}).
\begin{proof}
\; First remark that
$\displaystyle
\chi_{\mathcal{O}}=\; \frac{1}{|A|}\sum_{\mathcal{P} \in \,_G\! \backslash \hat{A}}\varPhi(\mathcal{P},\mathcal{O})\sum_{\xi \in \mathcal{P}}\xi,
$
because of $\chi_{\mathcal{O}}=\bar{\mathcal{F}}(\mathcal{F} \chi_{\mathcal{O}})$ and (\ref{me}),(\ref{barme}),(\ref{ghuhen-henkan}) and(\ref{ghuhen-gyaku-henkan}).
Therefore
\begin{eqnarray}\nonumber
\begin{split}
\{(&a_1,\cdots, a_k) \in \mathcal{O}_1 \times \cdots \times \mathcal{O}_k \mid a_1+ \cdots +a_k \in \mathcal{O}\}
=\; \sum_{a_1 \in \mathcal{O}_1}\dots \sum_{a_k \in \mathcal{O}_k} \chi_{\mathcal{O}}(a_1+\cdots +a_k) \\
&=\; \frac{1}{|A|}\sum_{a_1 \in \mathcal{O}_1}\dots \sum_{a_k \in \mathcal{O}_k} 
\sum_{\mathcal{P} \in \,_G\! \backslash \hat{A}}\varPhi(\mathcal{P},\mathcal{O})\sum_{\xi \in \mathcal{P}}\xi(a_1+\cdots +a_k)
\\
&
=\; \frac{1}{|A|} \sum_{\mathcal{P} \in \,_G\! \backslash \hat{A}}\varPhi(\mathcal{P},\mathcal{O})\sum_{\xi \in \mathcal{P}}\sum_{a_1 \in \mathcal{O}_1}\xi(a_1) \dots \sum_{a_k \in \mathcal{O}_k}\xi(a_k)
=\; \frac{1}{|A|} \sum_{\mathcal{P} \in \,_G\! \backslash \hat{A}}\varPhi(\mathcal{P},\mathcal{O}) \overline{\varPhi(\mathcal{P},\mathcal{O}_1) \cdots \varPhi(\mathcal{P},\mathcal{O}_k)}\,|\mathcal{P}|.
\end{split}
\end{eqnarray}

\end{proof}


\subsection{Relations with zonal spherical functions}
 \label{1-3}

The group-invariant Fourier transformation is related with zonal spherical functions in harmonic analysis on a finite symmetric space.
In this subsection, we summarize this relationship refering to \cite{Mac}, \cite{T}. 

In general, a pair $(\mathbb{G}, G)$ of a finite group $\mathbb{G}$ and its subgroup $G$ is called Gelfand pair when
any irreducible representation $(\tau, V)$ of $\mathbb{G}$ has at most $1$-dimensional $G$-invariant subspace $V^G= \{v \in V \mid \forall g \in G,\; \tau(g)v = v  \}$.
It is equivalent to saying that, the $\mathbb{G}$-representation $\mathbb{C}[\mathbb{G}/G]$ is multiplicity free, that is, decomposed into irreducible representations with multiplicity 1, where $\mathbb{C}[\mathbb{G}/G]$ is 
a subrepresentation of left regular representation of $\mathbb{G}$.
For a Gelfand pair $(\mathbb{G}, G)$,
an irreducible $\mathbb{G}$-representation $(\tau, V)$ such that dim$V^G=1$ is called {\bf spherical representation}.
It is uniquely realized as a subrepresentation $W_{\tau}$ in $\mathbb{C}[\mathbb{G}/G]$.
There is a unique function $\omega_{\tau} \in (W_{\tau})^G$ with $\omega_{\tau}(1_{\mathbb{G}})=1$.
It is called {\bf zonal spherical function} corresponding to $(\tau,V)$.
More specifically, it is given by the formula
\begin{equation}
 \label{kyu}
\omega_{\tau}(x)\,=\, (\tau(x^{-1})v_0,\; v_0) \qquad (x \in \mathbb{G}),
\end{equation}
where $(\;,\;)$ is an inner product on $V$ by which it is unitary, and $v_0 \in V^G$ is a unique element such that $(v_0,\;v_0)=1$.

Now we return to our setting that a finite group $G$ acts on a finite Abelian group $A$.
If we put $\mathbb{G}=A \rtimes G$, a semi-direct product corresponding to the action, then $(\mathbb{G},G)$ is a Gelfand pair, or more strongly, a weakly symmetric Gelfand pair.
In fact, if we set a group-automorphism $\sigma:\mathbb{G} \rightarrow \mathbb{G}$ by $(b,g) \mapsto (-b,g)$, then $x^{-1} \in G \sigma(x) G$ holds for all $x \in \mathbb{G}$.
Refer \cite{T} for that this condition implies it is a Gelfand pair.
Identifing the homogeneous space $\mathbb{G}/G$ with $A$, the transitive $\mathbb{G}$-action $\tau$ on $A$ is given by
\begin{equation}
 \label{boldg-act}
\tau(b,g)a=\; b+\rho(g)a \qquad (\;(b,g)\in \mathbb{G},\; a \in A),
\end{equation}
and corresponding $\mathbb{G}$-representation $(\tau, \mathbb{C}[A])$ is defined.
The decomposition of it into spherical representations is given as follows:

\begin{pp}
 \label{kyupp}
\; $\mathbb{C}[A]$ is decomposed into spherical representations of $(\mathbb{G}, G)$ as
\begin{equation}
 \label{bunkai}
\mathbb{C}[A]=\; \bigoplus_{\mathcal{P} \in\,_G\! \backslash \hat{A}} \, W_{\mathcal{P}},\quad \text{where}\quad W_{\mathcal{P}}=\; \rm{Span} \mathcal{P}=\; \{ \sum_{\xi \in \mathcal{P}} c_{\xi} \xi|\; c_{\xi} \in \mathbb{C} \}.
\end{equation}
The zonal spherical function $\omega_{\mathcal{P}}$ corresponding to $W_{\mathcal{P}}$ is given by
\begin{equation}
 \label{omegap}
\omega_{\mathcal{P}}=\; \frac{1}{|\mathcal{P}|} \sum_{\xi \in \mathcal{P}} \xi.
\end{equation}
\end{pp}

\begin{proof} \; As well known, characters $\hat{A}$ forms a basis of $\mathbb{C}[A]$.
Thus the decomposition (\ref{bunkai}) holds as vector spaces.
For $\forall \xi \in \mathcal{P}$, $b \in A$ and $g \in G$, 
$$
\tau(b,1_G)^{-1}\xi=\; \xi \circ \tau(b,1_G)=\; \xi(b+\,\cdot \,)=\; \xi(b) \xi \in W_{\mathcal{P}},
$$
$$
\tau(0,g)^{-1}\xi=\; \xi \circ \tau(0,g)=\; \xi \circ \rho(g) \in\mathcal{P} \subset W_{\mathcal{P}},
$$
which implies $W_{\mathcal{P}}$ is a $\mathbb{G}$-subrepresentation of $(\tau, \mathbb{C}[A])$.

Now the function $\omega_{\mathcal{P}}$ defined by (\ref{omegap}) obviously belongs to $(W_{\mathcal{P}})^G$.
Let $\displaystyle \varphi = \sum_{\xi \in \mathcal{P}} c_{\xi}\xi$ ($c_{\xi} \in \mathbb{C}$) be an arbitrary element of $(W_{\mathcal{P}})^G$.
Then  
$$
\varphi
=\; \frac{1}{|G|} \sum_{g \in G} \varphi \circ \rho(g)
=\; \frac{1}{|G|} \sum_{g \in G} \sum_{\xi \in \mathcal{P}} c_{\xi} \xi \circ \rho(g)
=\; \frac{1}{|G|} \sum_{\xi \in \mathcal{P}} c_{\xi} \sum_{g \in G} \xi \circ \rho(g)
=\; (\sum_{\xi \in \mathcal{P}} c_{\xi}) \frac{1}{|\mathcal{P}|}\sum_{\eta \in \mathcal{P}}\eta
=\; (\sum_{\xi \in \mathcal{P}} c_{\xi})\, \omega_{\mathcal{P}}.
$$
Thus $ \varphi \in \mathbb{C} \omega_{\mathcal{P}}$.
It concludes that $W_{\mathcal{P}}$ is a spherical representation and $\omega_{\mathcal{P}}$ is the zonal spherical function correponding to it.
\end{proof}
By (\ref{me-barme}) and (\ref{omegap}), a relation between canonical matrix and zonal spherical functions is induced:
\begin{col}
 \label{me-kyu}
\quad For orbits $\mathcal{O} \in\,_G\! \backslash A$ and $\mathcal{P} \in\,_G\! \backslash \hat{A}$,
$$
\omega_{\mathcal{P}}(\mathcal{O})\;=\; \frac{1}{|\mathcal{O}|}\overline{\varPhi(\mathcal{P},\mathcal{O})}.
$$
\end{col}


\subsection{Correspondence of orbits on a finite vector space}\label{1-4}

Let $\mathbb{F}$ be a finite field.
In this subsection we assume $A$ has structure of a vector space over $\mathbb{F}$ and the action $\rho$ is as $\mathbb{F}$-linear isomorphisms.
Let $\theta \in \hat{\mathbb{F}}$ be a non-trivial additive character of $\mathbb{F}$ and $\langle \;\; |\;\; \rangle$ be a non-degenerate symmetric bilinear form on $A$.
For $a \in A$, define $\theta_{a} \in \hat{A}$ by 
\begin{equation}
 \label{taiou}
\theta_a(b)=\; \theta(\langle a | b \rangle)\qquad (b \in A).
\end{equation}
Then it is easy to see that the correspondence
$\Theta: A \rightarrow \hat{A},\; a \mapsto \theta_a$
is a group isomorphism.
Next we define a property of $G$-actions which gets along well with this correspondence. 
\begin{dfn}
 \label{adjf}
The action $\rho$ is called {\bf adjoint-free} when there exists a map $t:G \rightarrow G,\; g \mapsto {}^t \! g$ such that the adjoint map of $\rho (g)$ according to the form $\langle \;\; |\;\; \rangle$ is given by $\rho (^{t} \! g)$ for all $g \in G$.
\end{dfn}

If the action $\rho$ is adjoint-free, the correspondence $\Theta$ preserves $G$-orbits of $A$ and $\hat{A}$, that is, two elements $a$ and $b$ of $A$ belong to a same orbit if and only if so do $\theta_a$ and $\theta_b$ in $\hat{A}$.
In fact, 
we can verify $\hat{\rho}(g^{-1})\theta_a=\theta_{\rho(^t\! g)a}$, which implies.
Thus using a parameter set $\Lambda$, we can write $_G\!\backslash A=\;\{\mathcal{O}(\lambda) \mid \lambda \in \Lambda \}$ and $_G\!\backslash \hat{A}=\;\{\mathcal{P}(\mu) \mid \mu \in \Lambda \}$, where $\mathcal{P}(\mu) = \Theta(\mathcal{O}(\mu))$.
Let $|\lambda| := |\mathcal{O}(\lambda)|=|\mathcal{P}(\lambda)|$ for $\lambda \in \Lambda$.
Under this notation, the canonical matrix $\varPhi=\left( \varPhi(\mu,\lambda)\right)$ of $\mathcal{F}$ has a symmetry:
\begin{equation}
 \label{me-taisho}
\frac{1}{|\lambda|} \varPhi(\mu, \lambda)=\; \frac{1}{|\mu|} \varPhi(\lambda, \mu) \qquad (\lambda, \mu \in \Lambda).
\end{equation}
In fact, since $\theta_a(b)=\theta_b(a) \; (a,b \in A)$ by symmetry of $\langle \; |\; \rangle$, we have
$\displaystyle
|\mu| \varPhi(\mu, \lambda)
= \sum_{a \in \mathcal{O}(\mu)} \sum_{b \in \mathcal{O}(\lambda)} \theta_a(b)
= \sum_{b \in \mathcal{O}(\lambda)} \sum_{a \in \mathcal{O}(\mu)} \theta_b(a)
= |\lambda| \varPhi(\lambda, \mu).
$
Let $\bar{\varPhi}=\left( \bar{\varPhi}(\lambda,\mu)\right)$ be the canonical matrix of $|A|\bar{\mathcal{F}}$.
Then by (\ref{me-barme}) and (\ref{me-taisho}), we have
\begin{equation}
 \label{barme-v}
\bar{\varPhi}(\lambda, \mu)=\; \overline{\varPhi(\lambda, \mu)} \qquad (\lambda, \mu \in \Lambda).
\end{equation}
Especially, assume all values of $\varPhi$ are real (For example, a condition $\calo = -\calo$ for $\forall \calo \in \,_G\backslash A$ is sufficient for it by (\ref{me})).
Then we have $\varPhi = \bar{\varPhi}$, and $\frac{1}{\sqrt{|A|}} \calf = \sqrt{|A|} \bar{\calf}$ by (\ref{ghuhen-henkan}) and (\ref{ghuhen-gyaku-henkan}). 
So we have a reversible formula as follows:
For a function $\varphi \in \mathbb{C}[\Lambda]$, let $\hat{\varphi}= \frac{1}{\sqrt{|A|}} \calf \varphi \in \mathbb{C}[\Lambda]$, that is,
\begin{equation}
 \label{sin-hanten}
\hat{\varphi}(\mu)=\; \frac{1}{\sqrt{|A|}} \sum_{\lambda \in \Lambda} \varPhi(\mu, \lambda) \varphi(\lambda)\qquad(\mu \in \Lambda),
\end{equation}
then $\hat{\hat{\varphi}}$ holds.

Now let $\Sigma$ be a new parameter set such that $|\Sigma|=|\Lambda|$
and  $P=(P(\lambda, \sigma))_{\lambda \in \Lambda, \sigma \in \Sigma}$ be a non-singular complex matrix.
Let's consider a change of bases from canonical basis $\{\chi_\lambda=\chi_{\mathcal{O}(\lambda)} \mid \lambda \in \Lambda \}$ to a new basis $\{e_\sigma \mid \sigma \in \Sigma\}$ of $\mathbb{C}[ _G\!\backslash A]$ by $P$, and at the same time $\{ \chi_\mu=\chi_{\mathcal{P}(\mu)} \mid \mu \in \Lambda \}$ to $\{f_\varepsilon \mid \varepsilon \in \Sigma\}$ of $\mathbb{C}[ _G\!\backslash \hat{A}]$
by complex conjugate $\overline{P}$. 
That is, define
\begin{equation}
e_\sigma = \sum_{\lambda \in \Lambda} P(\lambda,\sigma) \chi_\lambda,\quad
f_\varepsilon = \sum_{\mu \in \Lambda} \overline{P(\mu, \varepsilon)} \chi_\mu.
\end{equation}
Let $\Psi=(\Psi(\varepsilon, \sigma))$ and  $\bar{\Psi}=(\bar{\Psi}(\sigma, \varepsilon))$ be the matrices of $\mathcal{F}$ and $|A|\bar{\mathcal{F}}$ respectively on these new bases. 
Then since $\Psi = \overline{P^{-1}} \varPhi P$ and $\bar{\Psi} = P^{-1} \bar{\varPhi} \overline{P}$,
we get
\begin{equation}
 \label{bar-change-me}
\bar{\Psi}(\sigma, \varepsilon)=\; \overline{\Psi(\sigma, \varepsilon)} \qquad (\sigma, \varepsilon \in \Sigma).
\end{equation}


\section{Relations between two group-invariant Fourier transformations}

Also in this section, we think of a pair $(A, G)$ of a finite Abelian group $A$ and a finite group $G$ with an action $\rho$ on $A$ as group isomorphisms.
Adding it, let's think of another pair $(B, H)$ with an action $\tau$ in the same relation.
Here we describe relations between two group-invariant Fourier transformations
$\mathcal{F}_A:\mathbb{C}[_G\!\backslash A] \rightarrow \mathbb{C}[_G\!\backslash \hat{A}]$
and
$\mathcal{F}_B:\mathbb{C}[_H\!\backslash B] \rightarrow \mathbb{C}[_H\!\backslash \hat{B}]$.
Let $\pi: A \rightarrow B$ be a group homomorphism.


\subsection{A commutative diagram}\label{2-1}

Firstly we don't assume group actions and just consider
$\mathcal{F}_A: \mathbb{C}[A] \rightarrow \mathbb{C}[\hat{A}]$ and $\mathcal{F}_B: \mathbb{C}[B] \rightarrow \mathbb{C}[\hat{B}]$.
Let's fix an arbitrary character $\zeta \in \hat{A}$ which we call {\bf intersection character}, and define a transformation $\pi^{\zeta}_{*}: \mathbb{C}[A] \rightarrow \mathbb{C}[B]$ by
\begin{equation}
 \label{osidasi}
\pi^{\zeta}_{*}\varphi(b)=\; \sum_{a \in \pi^{-1}(b)} \overline{\zeta(a)} \varphi(a)\qquad (\varphi \in \mathbb{C}[A],\; b \in B).
\end{equation}
On the other hand, define $\hat{\pi}_{\zeta}: \hat{B} \rightarrow \hat{A}$ by
\begin{equation}
 \label{hatpi}
\hat{\pi}_{\zeta}\xi =\; (\xi \circ \pi)\cdot \zeta \qquad (\xi \in \hat{B}).
\end{equation}
Let $\hat{\pi}_{\zeta}^*$ be the pull-back of $\hat{\pi}_{\zeta}$, that is,    $\hat{\pi}_{\zeta}^*: \mathbb{C}[\hat{A}] \rightarrow \mathbb{C}[\hat{B}],\; \hat{\pi}_{\zeta}^* \psi=\; \psi \circ \hat{\pi}_{\zeta}$.
Then
\begin{pp}
\begin{equation}
\hat{\pi}_{\zeta}^* \circ \mathcal{F}_A =\; \mathcal{F}_B \circ \pi^{\zeta}_{*},
\end{equation}
that is, the following diagram is commutative:
\begin{equation}
 \label{d1}
\begin{CD}
\mathbb{C}[A] @>{\mathcal{F}_A}>>\mathbb{C}[\hat{A}] \\
@V{\pi^{\zeta}_{*}}V{\qquad\;\;\circlearrowright}V  @VV{\hat{\pi}_{\zeta}^*}V \\
\mathbb{C}[B] @>>{\mathcal{F}_B}> \mathbb{C}[\hat{B}]
\end{CD}
\end{equation}
\end{pp}
\begin{proof}
For $\forall \varphi \in \mathbb{C}[A]$, $\forall \xi \in \hat{B}$,\;
$\displaystyle
\hat{\pi}_{\zeta}^*(\mathcal{F}_A \varphi)(\xi)
=\; \mathcal{F}_A\varphi(\hat{\pi}_{\zeta}\xi)
=\; \sum_{a \in A} \overline{\xi(\pi(a))} \overline{\zeta(a)} \varphi(a)
=\; \sum_{b \in B} \sum_{a \in \pi^{-1}(b)} \overline{\xi(b)} \overline{\zeta(a)} \varphi(a) \\
=\; \sum_{b \in B} \overline{\xi(b)} \pi_*^{\zeta}\varphi(b)
=\; \mathcal{F}_B(\pi_*^{\zeta}\varphi)(\xi)
$.
\end{proof}

Next we intend to restrict the diagram (\ref{d1}) to the group-invariant functions.
Notice that following three conditions are equivalent:
$$
(\natural)\; \begin{cases} \vspace{0.1in}
(\natural 1) & \varphi \in \mathbb{C}[_G\!\backslash A] \; \Rightarrow \; \pi^{\zeta}_* \varphi \in \mathbb{C}[_H\!\backslash B].  \\ \vspace{0.1in}
(\natural 2) & \psi \in \mathbb{C}[_G\!\backslash \hat{A}] \; \Rightarrow \; \hat{\pi}_{\zeta}^* \psi \in \mathbb{C}[_H\!\backslash \hat{B}].\\
(\natural 3) & \text{ If $\xi,\; \xi' \in \hat{B}$ are in a same $H$-orbit, then $\hat{\pi}_{\zeta}\xi$ and $\hat{\pi}_{\zeta}\xi'$ are in a same $G$-orbit.}
\end{cases}
$$
In fact, the commutative diagram (\ref{d1}) implies the equivalence of $(\natural 1)$ and $(\natural 2)$. And the equivalence of $(\natural 2)$ and $(\natural 3)$ is easily verified.
When this equivalent condition $(\natural)$ holds, we can re-write diagram (\ref{d1}) as follows:

\begin{equation}
 \label{d2}
\begin{CD}
\mathbb{C}[_G\!\backslash A] @>{\mathcal{F}_A}>>\mathbb{C}[_G\!\backslash \hat{A}] \\
@V{\pi^{\zeta}_{*}}V{\qquad\quad\circlearrowright}V  @VV{\hat{\pi}_{\zeta}^*}V \\
\mathbb{C}[_H\!\backslash B] @>>{\mathcal{F}_B}> \mathbb{C}[_H\!\backslash \hat{B}]
\end{CD}
\end{equation}
We give a sufficient condition for $(\natural)$ as a reference.

\begin{pp}
 \label{flat1}
\; Assume that there exists a group homomorphism $\iota: H \rightarrow G $, and following two conditions hold:
$$
\begin{cases} \vspace{0.1in}
(\flat 1)& \forall h \in H,\; \tau(h) \circ \pi=\; \pi \circ \rho(\iota(h)).\\
(\flat 2)& \forall h \in H,\; \exists  g \in G\; \text{such that}\quad
\pi \circ \rho(g)=\; \pi\quad \text{and}\quad \zeta \circ \rho(g) =\; \zeta \circ \rho(\iota(h)).
\end{cases}
$$
Then the condition $(\natural)$ holds.
\end{pp}

\begin{proof}
Let's prove $(\natural 3)$.
Let $\xi, \xi' \in \hat{B}$ and $\xi'= \hat{\tau}(h)\xi,\; \exists h \in H$.
Take $g \in G$ for this $h$ in the assumption $(\flat 2)$.
Then,
$\displaystyle
\hat{\rho}(\iota(h) g^{-1})\; \hat{\pi}_{\zeta}\xi
=\;\left(\xi \circ \pi \circ \rho(g) \circ \rho(\iota(h^{-1}))\right)\cdot \left( \zeta \circ \rho(g) \circ \rho(\iota(h^{-1})) \right)
=\; \left( \xi \circ \tau(h^{-1}) \circ \pi \right)\cdot \zeta
=\; \hat{\pi}_{\zeta}\xi'.
$
\end{proof}


\subsection{Relations between canonical matrix elements}

In this subsection we assume the condition $(\natural)$.
So we have the commutative diagram (\ref{d2}).
By $(\natural 3)$, we can regard 
\begin{equation}
 \label{hatpi2}
\hat{\pi}_{\zeta}: \;_H\!\backslash \hat{B} \rightarrow \;_G\!\backslash \hat{A}
\end{equation}
as a map between orbits.
Let $\tilde{\mathcal{R}}$ denote $\hat{\pi}_{\zeta}(\mathcal{R})$ for $\mathcal{R} \in \,_H\!\backslash \hat{B}$.
Now we intend to see diagram (\ref{d2}) under canonical bases of each spaces.
$\mathcal{F}_A$ and $\mathcal{F}_B$ are expressed by canonical matrices $\varPhi_A=(\varPhi_A(\mathcal{P}, \mathcal{O}))$ and $\varPhi_B=(\varPhi_B(\mathcal{R}, \mathcal{Q}))$ defined in sec.\ref{1-2}.
Let $E^{\zeta}=(E^{\zeta}(\mathcal{Q}, \mathcal{O}))$ be the matrix of $\pi_*^{\zeta}$ on canonical bases $\{ \chi_{\mathcal{O}} \mid \mathcal{O} \in \,_G\!\backslash A \}$ and $\{ \chi_{\mathcal{Q}} \mid \mathcal{Q} \in \,_H\!\backslash B \}$.
And let $\Delta_{\zeta}=(\Delta_{\zeta}(\mathcal{R}, \mathcal{P}))$ be the matrix of $\hat{\pi}_{\zeta}^*$ on $\{ \chi_{\mathcal{P}} \mid \mathcal{P} \in \,_G\!\backslash \hat{A} \}$ and $\{ \chi_{\mathcal{R}} \mid \mathcal{R} \in \,_H\!\backslash \hat{B} \}$.
Then by definition,
\begin{equation}
 \label{Delta}
\Delta_{\zeta}(\mathcal{R}, \mathcal{P})=\; \hat{\pi}^*_{\zeta} \chi_{\mathcal{P}}(\mathcal{R})=\; \chi_{\mathcal{P}}(\tilde{\mathcal{R}})=\; \delta_{\tilde{\mathcal{R}}, \mathcal{P}}.
\end{equation}
In other words, $\Delta_{\zeta}$ is a matrix such that each row has just one element $1$ and others $0$
($\mathcal{R}$-th row has $1$ in its $\tilde{\mathcal{R}}$-th column only).
By diagram (\ref{d2}), $\Delta_{\zeta} \varPhi_A= \varPhi_B E^{\zeta}$ holds.
Let's compare $(\mathcal{R}, \mathcal{O})$-entries of both sides.
By (\ref{Delta}), the entry of left-hand side is
$\displaystyle
\sum_{\mathcal{P} \in \,_G\!\backslash\hat{A}} \delta_{\tilde{\mathcal{R}}, \mathcal{P}} \varPhi_A(\mathcal{P}, \mathcal{O})=\; \varPhi_A(\tilde{\mathcal{R}}, \mathcal{O}).
$
Thus the following theorem is proved:

\begin{thrm}
 \label{main1}
\; For $\mathcal{O} \in \,_G\! \backslash A$ and $\mathcal{R} \in \,_H\! \backslash \hat{B}$,
$$
\varPhi_{A}(\tilde{\mathcal{R}},\mathcal{O})\;=\; \sum_{\mathcal{Q} \in \,_H\! \backslash B}\, E^{\zeta}(\mathcal{Q},\mathcal{O}) \,
\varPhi_{B}(\mathcal{R},\mathcal{Q}).
$$
\end{thrm}
\noindent
Remember $\tilde{\mathcal{R}}=\hat{\pi}_{\zeta}(\mathcal{R})$.
So remark that this theorem gives the value of $\varPhi_A(\mathcal{P}, \mathcal{O})$ only when $\mathcal{P} \in \,_G\! \backslash \hat{A}$ is gotten as an image of $\hat{\pi}_{\zeta}$ in (\ref{hatpi2}).
By the way, let us give a remark for calculating $E^\zeta(\mathcal{Q}, \mathcal{O})$.
By definition,
$\displaystyle
E^\zeta(\mathcal{Q}, \mathcal{O})
= \pi_*^{\zeta}\chi_{\mathcal{O}}(\mathcal{Q})
= \sum_{a \in \pi^{-1}(b)}\chi_{\mathcal{O}}(a) \overline{\zeta(a)}
\quad(b \in \mathcal{Q}).
$
Thus when $\pi^{-1}(\mathcal{Q})=\phi$, $E^\zeta(\mathcal{Q}, \mathcal{O})=0$.
Otherwise for any element $a_0 \in \pi^{-1}(\mathcal{Q})$,
\begin{equation}
 \label{E1}
E^{\zeta}(\mathcal{Q},\mathcal{O})
=\; \overline{\zeta(a_0)} \sum_{a \in \ker(\pi)} \chi_{\mathcal{O}}(a_0+a) \overline{\zeta(a)}.
\end{equation}
It induces a little useful property of $E^\zeta$ as follows.
\begin{pp}
 \label{Erow}
\; If the intersection character $\zeta$ is non-trivial on $\ker(\pi)$, then for $\forall \mathcal{Q} \in \,_H\!\backslash B$,
$$
\sum_{\mathcal{O} \in \,_G\! \backslash A}\, E^\zeta(\mathcal{Q}, \mathcal{O})=0.
$$
\end{pp}
\begin{proof}
We can assume $\pi^{-1}(\mathcal{Q})\neq \phi$.
Let $a_0 \in \pi^{-1}(\mathcal{Q})$.
Then by (\ref{E1}),
$\displaystyle
\sum_{\mathcal{O} \in \,_G\! \backslash A}\,E^{\zeta}(\mathcal{Q},\mathcal{O})
=\; \overline{\zeta(a_0)} \sum_{a \in \ker(\pi)} \overline{\zeta(a)}=\;0.
$
\end{proof}

Next let's use the diagram (\ref{d2}) with inverse Fourier transformations $\bar{\mathcal{F}}_A$ and $\bar{\mathcal{F}}_B$.
Let $\bar{\varPhi}_A=(\bar{\varPhi}_A(\mathcal{O}, \mathcal{P}))$ and $\bar{\varPhi}_B=(\bar{\varPhi}_B(\mathcal{Q}, \mathcal{R}))$ be canonical matrices of $|A| \bar{\mathcal{F}}_A$ and $|B| \bar{\mathcal{F}}_B$ respectively.
Then by the diagram (\ref{d2}), $\displaystyle \frac{1}{|A|} E^\zeta \bar{\varPhi}_A=\frac{1}{|B|}\bar{\varPhi}_B \Delta_\zeta$ holds.
Comparing $(\mathcal{Q}, \mathcal{P})$-entries of both sides, we get
\begin{equation}
\sum_{\mathcal{O} \in \,_G\!\backslash A} E^\zeta(\mathcal{Q}, \mathcal{O}) \bar{\varPhi}_A(\mathcal{O}, \mathcal{P})
=\; \frac{|A|}{|B|} \sum_{\mathcal{R}} \bar{\varPhi}_B(\mathcal{Q}, \mathcal{R}),
\end{equation}
where the sum in the right-hand side is over $\mathcal{R} \in \,_H\!\backslash \hat{B}$ such that $\tilde{\mathcal{R}}=\mathcal{P}$.
As special cases,
\begin{thrm}
 \label{main2}
\; For $\mathcal{P} \in \,_G\!\backslash \hat{A}$ and $\mathcal{Q} \in \,_H\!\backslash B$,\\
\rm{(1)} \quad if there are no $\mathcal{R} \in \,_H\!\backslash \hat{B}$ such that  $\mathcal{P}=\tilde{\mathcal{R}}$,
$$
\sum_{\mathcal{O} \in \,_G\!\backslash A} E^\zeta(\mathcal{Q}, \mathcal{O}) \bar{\varPhi}_A(\mathcal{O}, \mathcal{P})
=\;0.
$$
\rm{(2)} \quad  If there is just one $\mathcal{R} \in \,_H\!\backslash \hat{B}$ such that  $\mathcal{P}=\tilde{\mathcal{R}}$,
$$
\sum_{\mathcal{O} \in \,_G\!\backslash A} E^\zeta(\mathcal{Q}, \mathcal{O}) \bar{\varPhi}_A(\mathcal{O}, \mathcal{P})
=\; \frac{|A|}{|B|}\bar{\varPhi}_B(\mathcal{Q}, \mathcal{R}).
$$
\end{thrm}


\subsection{The case on a finite vector space}\label{2-4}

In this subsection, let's assume that $A$ and $B$ are vector spaces over a finite field $\mathbb{F}$, and $\pi: A \rightarrow B$ is linear.
Using a fixed non-trivial character $\theta \in \hat{\mathbb{F}}$ and
non-degenerate symmetric bilinear forms $\langle \; \mid \; \rangle_A$ and $\langle \; \mid \; \rangle_B$ on $A$ and $B$,
let's define correspondences $\Theta_A: A \rightarrow \hat{A},\, a \mapsto \theta_a$ and $\Theta_B: B \rightarrow \hat{B},\, b \mapsto \theta_b$ according to (\ref{taiou}) in sec.\ref{1-4}.
In this subsection, fix a element $e \in A$ and use $\theta_e$ as the intersection character (as $\zeta$ in previous subsections).
So let's write $\pi^e_*=\pi^{\theta_e}_*$ and $\hat{\pi}_e=\hat{\pi}_{\theta_e}$ respectively.
Let ${}^t\!\pi: B \rightarrow A$ be the adjoint map of $\pi$ according to the forms $\langle \; \mid \; \rangle_A$ and $\langle \; \mid \; \rangle_B$.
Then
\begin{equation}
 \label{hatpic}
\hat{\pi}_e(\theta_b)=\; \theta_{{}^t\!\pi(b)+e}\qquad (b \in B)
\end{equation}
holds.
In fact, for $a \in A$,
$
\hat{\pi}_e(\theta_b)a=
\theta_b(\pi(a))\cdot \theta_e(a)=
\theta \left(\langle b | \pi(a) \rangle_B + \langle e | a \rangle_A\right)
=\theta(\langle ^t\!\pi(b)+e\,|\,a \rangle_A)=
\theta_{{}^t\!\pi(b)+e}(a).
$
Assume that actions $\rho,\;\tau$ are linear and adjoint-free(Def.\ref{adjf}). 
So $\Theta_A$ and $\Theta_B$ preserve orbits, and
we can parametrize both $_G\!\backslash A$ and $_G\!\backslash \hat{A}$ by a set $\Lambda$, $_H\!\backslash B$ and $_H\!\backslash \hat{B}$ by a set $\Gamma$.
Then by (\ref{hatpic}), the condition $(\natural 3)$ in sec.\ref{2-1} can be re-written as follows:
$$
(\natural 3') \quad \text{ If $b,\; b' \in B$ are in a same $H$-orbit, then ${}^t\!\pi(b)+e$ and ${}^t\!\pi(b')+e$ are in a same $G$-orbit.}
$$
Let's assume the condition $(\natural)$ below.
Then we can regard 
\begin{equation}
 \label{hatpi-para}
\hat{\pi}_e: \Gamma \rightarrow \Lambda,\; \omega \mapsto \tilde{\omega}
\end{equation}
as a map between parameters corresponding to orbits.
That is, when $\omega \in \Gamma$ is a parameter of an orbit which includes an element $b$, then $\tilde{\omega}$ is the one which includes ${}^t\!\pi(b)+e$. 
And we can re-write Thm.\ref{main1} and Thm.\ref{main2} as folllows (for second one, also (\ref{barme-v}) is used).

\begin{thrm}
 \label{main3}
\; For $\lambda \in \Lambda$ and $\omega \in \Gamma$,
$$
\varPhi_{A}(\tilde{\omega},\lambda)\;=\; \sum_{\gamma \in \Gamma}\, E^{e}(\gamma, \lambda) \,
\varPhi_{B}(\omega, \gamma).
$$
\end{thrm}

\begin{thrm}
 \label{main4}
\; For $ \mu \in \Lambda$ and $\gamma \in \Gamma$,\\
\rm{(1)} \quad if there are no $\omega \in \Gamma$ such that  $\mu=\tilde{\omega}$,
$$
\sum_{\lambda \in \Lambda} \overline{E^e(\gamma, \lambda)} \varPhi_A(\lambda, \mu)
=\; 0.
$$
\rm{(2)} \quad  If there is just one $\omega \in \Gamma$ such that  $\mu=\tilde{\omega}$,
$$
\sum_{\lambda \in \Lambda} \overline{E^e(\gamma, \lambda)} \varPhi_A(\lambda, \mu)
=\; \frac{|A|}{|B|}\varPhi_B(\gamma, \omega).
$$
\end{thrm}

When we change bases, the commutative diagram (\ref{d2}) gives relations between the matrix elements of $\mathcal{F}_A$ and $\mathcal{F}_B$ on the changed bases also.
Then (\ref{bar-change-me}) would be useful.
See examples on $\S_n$, sec. 5.


\section{Recursions for ($q$)-Krawtchouk polynomials}\label{3}

In this section, we introduce two recursions concerning ($q$)-Krawtchouk polynomials which consistently occur in our examples.
Remark that \cite{Drec} investigates a particular case of this topic.

For any non-negative integer $N \geq 0$, let $f_N$ be a function defined on $\{0, 1, \dots, N\} \times \{0, 1, \dots, N\}$.
For convenience, we define $f_N(y,x)=0$ when $x \leq -1$ or $x \geq N+1$.
The first recursion for a family $\{f_N\}_{N \geq 0}$ is in a form such that
\begin{equation}
 \label{koutai}
f_N(y+1,x)=\; at^x f_{N-1}(y,x)-bt^{x-1}f_{N-1}(y,x-1)\qquad (N \geq 1,\;0 \leq y \leq N-1,\, 0 \leq x \leq N),
\end{equation}
where $a$, $b$ and $t$ are non-zero constants.
We call it `backward-shift Pascal recursion (BPR)'.
The second recursion is in a form such that
\begin{equation}
 \label{zensin}
f_N(y+1,x)-cf_N(y,x)=\; -dt^{2N-y-1}f_{N-1}(y,x-1)\qquad (N \geq 1,\;0 \leq y \leq N-1,\, 0 \leq x \leq N),
\end{equation}
where $c$, $d$ and $t$ are non-zero constants. 
We call it `forward-shift Pascal recursion (FPR)'.
Let $\sigma=f_0(0,0)$ and $O_N(x)=f_N(0,x)$.
Now we give the solution of these recursions.
First, If a family $\{f_N\}_{N \geq 0}$ satisfies FPR {\rm (\ref{zensin})}, then by induction we get
\begin{equation}
 \label{zensintoku}
f_N(y,x)=\; \sum_{i=0}^{y \wedge x} c^{y-i} (-d)^i t^{-\binom{i}{2}+2Ni-yi} \begin{bmatrix} y \\ i \end{bmatrix} _t O_{N-i}(x-i),
\end{equation}
where $\begin{bmatrix} N \\ x \end{bmatrix} _q$ is the Gaussian-polynomial defined in (\ref{gau}),
and we promise 
$
\begin{bmatrix} N \\ x \end{bmatrix} _1
$
means $\displaystyle \binom{N}{x}$.
If $\{f_N\}_{N \geq 0}$ satisfies both recursions,
we can determine $f_N$ uniquely as follows:

\begin{pp}
 \label{zenkasikitoku}
\; Assume a family $\{f_N\}_{N \geq 0}$ satisfies both BPR {\rm (\ref{koutai})} and FPR {\rm (\ref{zensin})}.

\;{\rm (1)}\;
When $t=1$ and $b = d$, then
\begin{eqnarray}\nonumber
 \label{onr1}
\begin{split}
&{\rm (i)}\quad  O_{N}(0)=\; \sigma \frac{a^N}{c^N},\quad O_N(x)=0\;(\text{for}\;x > 0),\text{and} \\
&{\rm (ii)}\quad  f_N(y,x)=\; \sigma \frac{a^{N-x}(-b)^x}{c^{N-y}}\binom{y}{x}\; (\text{for}\; x \leq y),\quad f_N(y,x)=\;0\; (\text{for}\; x > y).
\end{split}
\end{eqnarray}

\;{\rm (2)}\;
When $t=1$ and $b \neq d$, then
\begin{eqnarray}\nonumber
 \label{onr2}
\begin{split}
&{\rm (i)}\quad  O_{N}(x)=\; \sigma \frac{a^{N-x}}{c^N}  (d-b)^x\binom{N}{x},\; \text{and} \\
&{\rm (ii)}\quad  f_N(y,x)=\; c^y O_N(x)\, K_y(x; 1-\frac{b}{d}, N),
\end{split}
\end{eqnarray}

where $K_y(x; p,N)$ is the Krawtchouk polynomial defined in Def.\ref{kraw}.

\;{\rm (3)}\;
When $t \neq 1$, then
\begin{eqnarray}\nonumber
 \label{onr3}
\begin{split}
&{\rm (i)}\quad  O_{N}(x)=\; \sigma \frac{a^{N-x} (-b)^x}{c^N} t^{\binom{x}{2}} (\frac{d}{b}t^N; t^{-1})_x\begin{bmatrix} N \\ x \end{bmatrix} _t,\; \text{and} \\
&{\rm (ii)}\quad  f_N(y,x)=\; c^y O_N(x)\, K_y^{\Aff}(x; \frac{b}{d}t^{-N}, N; t),
\end{split}
\end{eqnarray}

where $K^{\Aff}_y(x; a,N; q)$ is the Affine $q$-Krawtchouk polynomial defined in Def.\ref{affqkraw}.

\end{pp}
\begin{proof}
In any cases, taking $y=0$ in BPR and FPR, we get
\begin{equation}
 \label{o-zenka}
cO_N(x)=\; at^{x}O_{N-1}(x) + (dt^{2N-1}-bt^{x-1}) O_{N-1}(x-1).
\end{equation}
Using it inductively, the values of $O_N(x)$ in each (i) is gotten.
Substituting it in (\ref{zensintoku}) and by simple calculation, we can get the values of $f_N(y,x)$ in each (ii).
\end{proof}


\section{Canonical matrices in some examples over a finite field}\label{four}
Throughout this section, $\mathbb{F}$ is a finite field of the order $q$, a power of a prime.
Here as concrete examples of previous sections, we deal with group-invariant Fourier transformations on $\mathbb{F}$-vector spaces $\mathbb{F}^n$, $\M_{n,m}(\mathbb{F})$ of $n \times m$-sized all matrices $(n\leq m)$ and $\A_n(\mathbb{F})$ of $n \times n$-sized alternating matrices with adequate group-actions.
In followings we omit writing coefficient field $\mathbb{F}$ like $\M_{n,m}=\M_{n,m}(\mathbb{F})$.
As a remark,  our all examples can be regarded as the sub-actions of $\G_n \times \G_m$ on $\M_{n,m}$ such that
\begin{equation}
\label{act2}
\rho(g,h) a = ga\,^t\!h\qquad \bigl( (g,h) \in \G_n \times \G_m,\; a \in \M_{n,m} \bigr),
\end{equation}
where and in the following contexts, $^t\!h$ denotes the transpose matrix of $h$.
One of our purposes here is to describe relations between two Fourier transformations on different sized vector spaces, like $\mathbb{F}^n$ and $\mathbb{F}^{n-1}$.
Then using this relations, we induce BPR (\ref{koutai}) and FPR (\ref{zensin}) of the canonical matrix elements, and solve them.

Fix a non-trivial additive character $\theta \in \hat{\mathbb{F}}$.

\subsection{The pair $(\mathbb{F}^n, (\mathbb{F}^\times)^n \rtimes \mathfrak{S}_n)$}
As a first example, we think of vector space $\mathbb{F}^n$ and an action $\rho$ of a wreath product $G_n=(\mathbb{F}^\times)^n \rtimes \mathfrak{S}_n$, that is, a semi-direct product group according to the permutation,
where $\mathbb{F}^\times$ is the multicative group of $\mathbb{F}$ and $\mathfrak{S}_n$ is symmetric group of degree $n$.
The action $\rho$ on $\mathbb{F}^n$ by $G_n$ is defined as follows:
\begin{equation}
\rho(c,\sigma)a=\; (c_1 a_{\sigma^{-1}(1)},\dots,c_n a_{\sigma^{-1}(n)}) \quad \quad \left(c=(c_1,\dots,c_n) \in (\mathbb{F}^\times)^n,
\;\sigma \in \mathfrak{S}_n,\; a=(a_1,\dots,a_n) \in \mathbb{F}^n\right).
\end{equation}
We consider $G_n$-invariant Fourier transformation $\mathcal{F}_n:\mathbb{C}[\,_{G_n}\! \backslash \mathbb{F}^n] \rightarrow \mathbb{C}[\,_{G_n}\! \backslash \hat{\mathbb{F}^n}]$.
Define the weight of an element $a=(a_1, \dots, a_n) \in \mathbb{F}^n$ by $\wt a = \sharp \{ i|\; a_i \neq 0 \}$.
Then $G_n$-orbits in $\mathbb{F}^n$ are characterized by it,
that is,
\begin{equation}
\,_{G_n}\! \backslash \mathbb{F}^n=\; \{\mathcal{O}(r) \mid 0 \leq r \leq n\}, \qquad \mathcal{O}(r) =\mathcal{O}_n(r) = \{a \in \mathbb{F}^n|\; \wt a = r \}.
\end{equation}
On the other hand, using $\theta$ (a non-trivial additive character of $\mathbb{F}$) and an ordinary non-degenerate symmetric bilinear form on $\mathbb{F}^n$ such as $\langle a|b \rangle=\, \sum_{j=1}^n a_j b_j\;(a,\!b \in \mathbb{F}^n)$, we correspond $\mathbb{F}^n$ and $\hat{\mathbb{F}^n}$ by the manner of (\ref{taiou}) in sec.\ref{1-4}.
That is, for $a \in \mathbb{F}^n$ define $\theta_a \in \hat{\mathbb{F}^n}$ by
$
\theta_a(b)=\, \theta(\langle a|b \rangle)\;\; (b \in \mathbb{F}^n).
$
Also check that the action is adjoint-free according to the form $\langle \;|\; \rangle$.
In fact we can verify that the adjoint map of $\rho(c,\sigma)$ is given by $\rho(\sigma^{-1}\cdot c, \sigma^{-1})\;(c \in (\mathbb{F}^\times)^n,\, \sigma \in \mathfrak{S}_n)$.
Therefore the orbits $_{G_n}\! \backslash \hat{\mathbb{F}^n}$ corresponds to $_{G_n}\! \backslash \mathbb{F}^n$ such that
\begin{equation}
_{G_n}\! \backslash \hat{\mathbb{F}^n}=\; \{\mathcal{P}(s) \mid 0 \leq s \leq n\},\qquad \mathcal{P}(s)= \mathcal{P}_n(s) = \{\theta_a \mid \wt a =s \}.
\end{equation}
Under this parametrization, the canonical matrix of $\mathcal{F}_n$ can be written as $\varPhi_n=(\varPhi_n(s,r))_{s,r=0}^n$.
Remark when $n=0$, we can regard $\mathbb{F}^0=0$ and $\varPhi_0(0,0)=1$.
And for example when $n=1$, we can easily know the values of $\varPhi_1$ in the following table:
\begin{table}[htb]
\begin{center}
\scalebox{1.1}[1.1]{
\begin{tabular}{|c||c|c|} \hline
${}_{s} \quad {}^{r}$ & 0 & 1 \\ \hline \hline
0 & {\footnotesize$1$} & {\footnotesize$q-1$} \\ \hline
1 & {\footnotesize$1$} & {\footnotesize$-1$} \\ \hline
\end{tabular}
}
\end{center}
\caption{The values $\varPhi_1(s,r)$.}
\end{table}

For an arbitrary $n$, the values of $\varPhi_n$ is given as follows: 

\begin{pp}
\label{diag-me}
\begin{equation}\nonumber
\varPhi_n(s,r)=\;(q-1)^r\, \binom{n}{r}\, K_s(r;\frac{q-1}{q},n)
\qquad (0 \leq s,\!r \leq n ).
\end{equation}
\end{pp}

In order to calculate it, we use a relation given in sec.2, between two group-invariant Fourier transformations $\mathcal{F}_n$ and $\mathcal{F}_{n-1}\;(n \geq 1)$.
Let $\pi:\mathbb{F}^n \rightarrow \mathbb{F}^{n-1}$ be a projection map $(a_1,\dots, a_n) \mapsto (a_1, \dots.a_{n-1})$.
Then remark that the adjoint map ${}^t\!\pi$ is an embedding map $(a_1, \dots.a_{n-1}) \mapsto (a_1, \dots.a_{n-1},0)$.
Let's take $e=(0, \dots, 0,1) \in \mathbb{F}^n$ and use $\theta_e$ as an intersection character.
So, define $\pi_*:=\pi_*^e=\pi_*^{\theta_e}$ and $\hat{\pi}:=\hat{\pi}_e =\hat{\pi}_{\theta_e}$ according to (\ref{osidasi}) and (\ref{hatpi}).
Especially
\begin{equation}
 \label{diag-osidasi}
\pi_* \varphi (b)=\;\sum_{a \in \pi^{-1}(b)} \overline{\theta_e(a)} \varphi(a)
=\; \sum_{w \in \mathbb{F}} \overline{\theta(w)} \varphi(b,w)\qquad(\varphi \in \mathbb{C}[\mathbb{F}^n],\; b \in \mathbb{F}^{n-1}).
\end{equation}
Also since
${}^t\!\pi (b)+e= (b,1)\;(\forall b \in \mathbb{F}^{n-1})$,
then $(\natural 3')$ in sec.\ref{2-4} holds.
So we have a commutative diagram (\ref{d2}) as follows:
\begin{equation}
 \label{d3}
\begin{CD}
\mathbb{C}[_{G_n}\!\backslash \mathbb{F}^n] @>{\mathcal{F}_n}>>\mathbb{C}[_{G_n}\!\backslash \hat{\mathbb{F}}^n] \\
@V{\pi_{*}}V{\qquad\qquad\;\;\circlearrowright}V  @VV{\hat{\pi}^*}V \\
\mathbb{C}[_{G_{n-1}}\!\backslash \mathbb{F}^{n-1}] @>>{\mathcal{F}_{n-1}}> \mathbb{C}[_{G_{n-1}}\!\backslash \hat{\mathbb{F}}^{n-1}]
\end{CD}
\end{equation}
And when we regard $\hat{\pi}:\{0,1,\dots,n-1\} \to \{0,1,\dots,n\}, v \mapsto \tilde{v}$ as a map between parameters correponding to orbits by (\ref{hatpi-para}), then $\tilde{v}=v+1$.
It implies that the matrix of $\hat{\pi}^*$ on canonical bases is given by
\begin{equation}
 \label{diagDelta}
\Delta=(\delta_{v+1,s})_{0\leq v \leq n-1,\; 0 \leq s \leq n}
\end{equation}
by (\ref{Delta}).
Now let's solve the matrix $E=(E(u,r))_{u,r}$ of $\pi_*$ on canonical bases.
For $0\leq u \leq n-1$ and $0 \leq r \leq n$, using any element $b \in \mathbb{F}^{n-1}$ such that $\wt b=u$ and $\chi_r=\chi_{\mathcal{O}(r)}$, by(\ref{diag-osidasi}),
\begin{equation}
 \label{diagE}
E(u,r)=\;\pi_* \chi_r(b)
=\; \sum_{w \in \mathbb{F}}\overline{\theta(w)}\chi_r(b,w)
=\; \begin{cases} 1& (r=u)\\-1&(r=u+1)\\0&(otherwise).\end{cases}
\end{equation}
Now everything is ready for theorems in sec.2.
First by Thm.\ref{main3},
$\displaystyle
\varPhi_n(v+1, r)
=\; \sum_{u=0}^{n-1} E(u,r) \varPhi_{n-1}(v, u),$
so we get
\begin{equation}
 \label{diag-koutai}
\varPhi_n(v+1, r)=\; \varPhi_{n-1}(v,r)-\varPhi_{n-1}(v,r-1)\qquad (0 \leq v \leq n-1,\; 0\leq r \leq n),
\end{equation}
where we interpret $\varPhi_{n-1}(v,-1)=\varPhi_{n-1}(v,n)=0$ for convenience.
And by Th.\ref{main4},
$\displaystyle
\sum_{r=0}^n \overline{E(u,r)} \varPhi_n(r,s)
=\; \frac{|\mathbb{F}^{n}|}{|\mathbb{F}^{n-1}|}\varPhi_{n-1}(u, s-1)$,
so
\begin{equation}
 \label{diag-zensin}
\varPhi_{n}(u+1,s)-\varPhi_{n}(u,s)=\; -q \varPhi_{n-1}(u,s-1) \qquad (0 \leq u \leq n-1,\; 0\leq s \leq n).
\end{equation}
So when we regard $f_N(y,x)=\varPhi_N(y,x)$, these give BPR(\ref{koutai}) and FPR(\ref{zensin}) in sec.\ref{3} of the case $a=b=c=1,\; d=q,\; t=1$ and $\sigma=\varPhi_0(0,0)=1$.
Therefore by Prop.\ref{zenkasikitoku}(2) we get
\begin{eqnarray}
 \label{diag-kidou}
|\mathcal{O}_n(r)|=\; (q-1)^r\binom{n}{r},\quad \text{and} \\
\varPhi_n(s,r)=\; |\mathcal{O}_n(r)|\, K_s(r; \frac{q-1}{q}, n),
\end{eqnarray}
which proved Prop.\ref{diag-me}.

As remarks we notice some properties of a canonical matrix element $\varPhi_n(s,r)$ from recursions (\ref{diag-koutai}) and (\ref{diag-zensin}).

\begin{col}
\label{seisitu1}
{\rm (1)}\quad $\varPhi_n(s,r)$ is a polynomial in $q$ over integers.

{\rm (2)}\quad As $q \rightarrow 1$, if $r \leq s$ then
$\displaystyle \varPhi_n(s,r) \rightarrow (-1)^r \binom{s}{r}$,
otherwise $\displaystyle \varPhi_n(s,r) \rightarrow 0$.
\end{col}
\noindent
\underline{Remark}. \;(2) implies that the reversible formula (\ref{sin-hanten}) given by this example is a $q$-analogue of the reversible formula in Prop.\ref{mori}.
\begin{proof}
(1)\;  It is inductively known by the recursions (\ref{diag-koutai}) and (\ref{diag-kidou}). 

(2)\; Since the limit of $\varPhi_n(s,r)$ as $q \rightarrow 1$ satisfies
BPR(\ref{koutai}) and FPR(\ref{zensin}) of the case $a=b=c=d=1,\; t=1$ and $\sigma=1$ in sec.\ref{3}, then by Prop.\ref{zenkasikitoku}(1) we get the result.
\end{proof}

Also by use of diagram(\ref{d3}), we can derive the generating function of $\varPhi_n(s,r)$ concerning $r$.

\begin{pp}
 \label{diag-bo}
Let $t$ be a variable. We have
$$
\sum_{r=0}^n \varPhi_n(s,r) t^r=\;(1-t)^s (1+(q-1)t)^{n-s}\qquad (0 \leq s \leq n).
$$
\end{pp}
\begin{proof}
Combining different sized $s$ diagrams such as (\ref{d3}), we get following commutative diagram,

\begin{equation}
 \label{d4}
\begin{CD}
\mathbb{C}[_{G_n}\!\backslash \mathbb{F}^n] @>{\mathcal{F}_n}>>\mathbb{C}[_{G_n}\!\backslash \hat{\mathbb{F}}^n] \\
@V{\prod^s(\pi_{*})}V{\qquad\qquad\;\;\circlearrowright}V @VV{\prod^s (\hat{\pi}^*)}V \\
\mathbb{C}[_{G_{n-s}}\!\backslash \mathbb{F}^{n-s}] @>>{\mathcal{F}_{n-s}}> \mathbb{C}[_{G_{n-s}}\!\backslash \hat{\mathbb{F}}^{n-s}]
\end{CD}
\end{equation}
where $\prod^s(\pi_{*})$ means a composition of $s$ maps $\pi_{*}$ in (\ref{diag-osidasi}) in each size.
$\prod^s(\hat{\pi}^*)$ is likewise.
Let $\displaystyle \prod^sE$ and $\displaystyle \prod^s\Delta$ be the matrices of $\prod^s(\pi_{*})$ and $\prod^s(\hat{\pi}^*)$ respectively.
By (\ref{diagE}), $E\left(t^i\right)_{i=0}^n=(1-t)\left(t^i\right)_{i=0}^{n-1}$ holds.
Thus inductively, 
\begin{equation}
 \label{papa}
\varPhi_{n-s} (\prod^sE)\left(t^i\right)_{i=0}^n=\;(1-t)^s\varPhi_{n-s}\left(t^i\right)_{i=0}^{n-s}
=\;(1-t)^s \left( \sum_{u=0}^{n-s} \varPhi_{n-s}(i,u)t^u \right)_{i=0}^{n-s}.
\end{equation}
On the other hand by (\ref{diagDelta}), we have
\begin{equation}
 \label{mama}
(\prod^s\Delta) \varPhi_n \left(t^i\right)_{i=0}^n
=\; (\prod^s\Delta) \left( \sum_{r=0}^n\varPhi_n(i,r)t^r \right)_{i=0}^n
=\; \left( \sum_{r=0}^n\varPhi_n(i,r)t^r \right)_{i=s}^n.
\end{equation}
Comparing first elements of (\ref{papa}) and (\ref{mama}), we get
$\displaystyle
\sum_{r=0}^n\varPhi_n(s,r)t^r=(1-t)^s \sum_{u=0}^{n-s} \varPhi_{n-s}(0,u)t^u$,
and substituting $\displaystyle \varPhi_{n-s}(0,u)=(q-1)^u \binom{n-s}{u}$ by (\ref{diag-kidou}), we get the statement.
\end{proof}

We can also derive the well known generating function of Krawtchouk polynomials\cite{Koe} as follows:\; Let $T$ be a variable,
\begin{equation}
 \label{kraw-bo}
\sum_{y=0}^N \binom{N}{y}K_y(x; p,N) T^y
=\;(1-\frac{1-p}{p}T)^x (1+T)^{N-x}\qquad (0 \leq x \leq N,\; p \in \mathbb{C} -\{0,1\}).
\end{equation}
In fact, since both sides of Prop.\ref{diag-bo} are polynomials in $q$ which is an arbitrary prime power, it holds for $\forall q \in \mathbb{C}$.
Thus substituting Prop.\ref{diag-me} and $q=\frac{1}{1-p}$ and then changing variable by $t=\frac{1-p}{p}T$, we get (\ref{kraw-bo}).

Next let's derive a kind of multi-orthogonality relation of $\varPhi_n(s,r)$ or Krawtchouk polynomials.
\begin{pp}
 \label{diag-cho}
\; Let $k \geq 1$ and $0 \leq r_1,\cdots, r_k,r \leq n$ such that $r_1+\cdots +r_k \leq r$, then
$$
\sum_{s=0}^n (q-1)^s\binom{n}{s} \varPhi_n(s,r_1)\cdots \varPhi_n(s,r_k)\varPhi_n(s,r)
=\; \delta_{r_1+\cdots +r_k, r}\;\frac{n!}{(n-r)!r_1! \cdots r_k!}(q-1)^r q^n.
$$
\end{pp}
\begin{proof}
We count the number $N:=\sharp \{ (a_1,\cdots,a_k) \in \mathcal{O}(r_1) \times \cdots \times \mathcal{O}(r_k) \mid a_1+ \cdots +a_k \in \mathcal{O}(r) \}.$
When $r_1+\cdots +r_k < r$, obviously $N=0$.
When $r_1+\cdots +r_k = r$, by combinatorial consideration,
$$
N=\; \binom{n}{r_1} \binom{n-r_1}{r_2}\cdots \binom{n-r_1-\cdots-r_{k-1}}{r_k}(q-1)^r=\; \frac{n!}{(n-r)!r_1! \cdots r_k!}(q-1)^r.
$$
Thus by Prop.\ref{taju}, we get the result.
\end{proof}

\begin{col}
 \label{kraw-cho}
Let $k \geq 1$ and $0 \leq y_1,\cdots, y_k,y \leq N$ such that $y_1+\cdots +y_k \leq y$, then
$$
\sum_{x=0}^N p^x(1-p)^{N-x}\binom{N}{x} K_{y_1}(x;p,N) \cdots K_{y_k}(x;p,N) K_{y}(x;p,N)=\; \delta_{y_1+\cdots +y_k, y}\;\frac{(N-y_1)!\cdots (N-y_k)! y!}{(N!)^k}\frac{(1-p)^y}{p^y}.
$$
\end{col}
\noindent
\underline{Remark}.\; When k=1, it becomes the orthogonality relation well known \cite{Koe} as follows:
$$
\sum_{x=0}^N p^x (1-p)^{N-x} \binom{N}{x} K_y(x;p,N) K_{y'}(x;p,N)=\;
\delta_{y,y'} \; \frac{(1-p)^y}{\binom{N}{y}\, p^y}.
$$
\begin{proof}
\; Since Prop.\ref{diag-cho} holds for $\forall q \in \mathbb{C}$, substituting Prop.\ref{diag-me} and $q=\frac{1}{1-p}$, it follows.
\end{proof}


\subsection{The pair $(\M_{n,m}, \G_n \times \G_m)$\;$(n \leq m)$}
In this subsection, let $n \leq m$ for simplicity. 
Let $G_{n,m} =\G_n \times \G_m$ and consider a pair $(\M_{n,m}, G_{n,m})$ with an action $\rho$ such that 
\begin{equation}
\label{mat-act}
\rho(g,h) a = ga\,^t\!h\qquad \bigl( (g,h) \in \G_n \times \G_m,\; a \in \M_{n,m} \bigr)
\end{equation}
(same as (\ref{act2})).
Obviously orbits are characterized by rank of the matrices, that is,
\begin{equation}
_{G_{n,m}}\! \backslash \M_{n,m} =\; \{\mathcal{O}(r) \mid 0 \leq r \leq n\}, \qquad \mathcal{O}(r) = \mathcal{O}_{n,m}(r) =  \{a \in \M_{n,m}|\; \rk a = r \}.
\end{equation}
On the other hand, the ordinary non-degenerate symmetric bilinear form $\langle \;|\; \rangle$ on $\M_{n,m}$ is given by
\begin{equation}
\label{form}
\langle a|b \rangle=\; \tr(\, {}^t \! ab) \quad (a, b \in \M_{n,m}).
\end{equation}
$\rho$ is adjoint-free according to this form. In fact the adjoint map of $\rho(g,h)$ is given by $\rho({}^t\!g,\;{}^t\!h)$.
Thus orbits in $(\M_{n,m})\,\hat{}\;$ correspond to those in $\M_{n,m}$ by the manner of (\ref{taiou}), that is,
\begin{equation}
_{G_{n,m}}\! \backslash (\M_{n,m})\,\hat{}\;=\; \{\mathcal{P}(s) \mid 0 \leq\; s\; \leq n\},\qquad \mathcal{P}(s)=\mathcal{P}_{n,m}(s)=\{\theta_a \mid \rk a =\,s \}.
\end{equation}
Let $\mathcal{F}_{n,m}: \mathbb{C}[_{G_{n,m}}\! \backslash \M_{n,m}] \rightarrow  \mathbb{C}[_{G_{n,m}}\! \backslash (\M_{n,m})\,\hat{}\;]$ be the group-invariant Fourier transformation,
and $\varPhi_{n,m}=(\varPhi_{n,m}(s,r))$ be the canonical matrix of it.
When $n=0$, we can regard $\M_{0,m}=0$ and $\varPhi_{0,m}(0,0)=1$ for $\forall m \geq 0$.
When $n=1$, the values of $\varPhi_{1,m}\;(m \geq 1)$ are in the following table:
\begin{table}[htb]
\begin{center}
\scalebox{1.1}[1.1]{
\begin{tabular}{|c||c|c|} \hline
${}_{s} \quad {}^{r}$ & 0 & 1 \\ \hline \hline
0 & {\footnotesize$1$} & {\footnotesize$q^m-1$} \\ \hline
1 & {\footnotesize$1$} & {\footnotesize$-1$} \\ \hline
\end{tabular}
}
\end{center}
\caption{The values $\varPhi_{1,m}(s,r)$.}
\end{table}

For an arbitrary $n$, we show
\begin{pp}
 \label{mat-me}
\begin{equation}\nonumber
\varPhi_{n,m}(s,r)
=\; (-1)^rq^{\binom{r}{2}} (q^m;q^{-1})_r \begin{bmatrix} n \\ r \end{bmatrix} _q \, K^{\Aff}_s(r; q^{-m}, n; q) \qquad (0 \leq s,\!r \leq n \leq m ).
\end{equation}
\end{pp}
\noindent
In order to calculate it, we use a relation between $\mathcal{F}_{n,m}$ and $\mathcal{F}_{n-1,m-1}$\,($1 \leq n \leq m$).
For convenience we use a notation
\begin{equation} 
\label{notation1}
(b\mid y, z, w):=\; \left(\begin{array}{cc}\\[-10pt]
b& y \\[0pt]
z& w  \end{array}\right)
\in \M_{n,m}
\qquad \text{for}\;\; b \in \M_{n-1,m-1},\; y \in \mathbb{F}^{n-1},\; z \in \mathbb{F}^{m-1} \text{and}\;\; w \in \mathbb{F}.
\end{equation}
Let $\pi:\M_{n,m} \rightarrow \M_{n-1,m-1}$ be a projection map $(b \mid y,z,w) \mapsto b$.
Then the adjoint map ${}^t\!\pi$ according to the forms (\ref{form}) is an embedding map $b \mapsto (b\mid 0,0,0)$.
Let's take an intersection character $\theta_e$ where $e=(0 \mid 0,0,1) \in \M_{n,m}$ and define $\pi_*=\pi_*^{\theta_e}$ and $\hat{\pi}=\hat{\pi}_{\theta_e}$ by (\ref{osidasi}) and (\ref{hatpi}).
Since ${}^t\!\pi(b) +e=(b\mid 0,0,1)\;(\forall b \in \M_{n-1,m-1})$, we can ensure that ($\natural 3'$) holds, $\tilde{v}=\hat{\pi}(v)=v+1$ where we see $\hat{\pi}:\{0,1,\dots,n-1\} \to \{0, 1, \dots, n\}$ by (\ref{hatpi-para}), and the matrix of $\hat{\pi}^*$ on canonical bases is
\begin{equation}
 \label{matDelta}
\Delta=(\delta_{v+1,s})_{0\leq v \leq n-1,\; 0 \leq s \leq n}
\end{equation}
(same as (\ref{diagDelta})).
A commutative diagram(\ref{d2}) here is as follows:
\begin{equation}
 \label{mat-d}
\begin{CD}
\mathbb{C}[_{G_{n,m}}\!\backslash \M_{n,m}] @>{\mathcal{F}_{n,m}}>>\mathbb{C}[_{G_{n,m}}\!\backslash (\M_{n,m})\,\hat{}\;] \\
@V{\pi_{*}}V{\qquad\qquad\qquad\quad\;\;\circlearrowright}V  @VV{\hat{\pi}^*}V \\
\mathbb{C}[_{G_{n-1,m-1}}\!\backslash \M_{n-1,m-1}] @>>{\mathcal{F}_{n-1,m-1}}> \mathbb{C}[_{G_{n-1,m-1}}\!\backslash (\M_{n-1,m-1})\,\hat{}\;]
\end{CD}
\end{equation}
Now let's calculate the matrix $E=(E(u,r))_{u,r}$ of $\pi_*$ on canonical bases.
The conclusion is 
\begin{equation}
 \label{matE}
E(u,r)=\; \begin{cases} q^u & (r=u)\\ -q^u & (r=u+1) \\ 0 & (otherwise) \end{cases}
\qquad\quad (0\leq u \leq n-1,\;0 \leq r \leq n).
\end{equation}
We prove it below.
We use
$b=
\left(\!\!\!\!
\begin{array}{ccc}
 \begin{array}{ccc}
 1&&\\[-10pt]
 & \!\!\!\!\!\ddots&\\[-5pt]
 &&\!\!\!\!\!1
 \end{array}
&&\\[0pt]
&& 
\end{array}
\right)
 \in \M_{n-1,m-1}
$, where the number of $1$ is $u$ and the blanks are filled by $0$,
as a representative of the orbit $\mathcal{O}_{n-1,m-1}(u)$.
Then,
\begin{equation}
 \label{pico}
E(u,r)= \pi_*\chi_r(b)= \sum_{a \in \pi^{-1}(b)}\overline{\theta_e(a)}\chi_r(a)
=\sum_{y \in \mathbb{F}^{n-1}} \sum_{z \in \mathbb{F}^{m-1}} \sum_{w \in \mathbb{F}} \overline{\theta(w)}\chi_r(b \mid y,z,w).
\end{equation}
Since $\rk (b \mid y,z,w)$ is $u,\;u+1$ or $u+2$, we have $E(u,r)=0\;(r \neq u,\,u+1,\;u+2)$.
Let's write $y=(y', y''),\; y' \in \mathbb{F}^u,\; y'' \in \mathbb{F}^{n-u-1}$, and $z=(z', z''),\; z' \in \mathbb{F}^u,\; z'' \in \mathbb{F}^{m-u-1}$.
Then $\rk(b \mid y,z,w)=u+2$ if and only if $y'' \neq 0$ and $z'' \neq 0$.
Since this condition doesn't depend on $w \in \mathbb{F}$, (\ref{pico}) becomes a multiple of $\sum_{w \in \mathbb{F}} \overline{\theta}(w)=0$, so we have $E(u,u+2)=0$.
Next when $r=u$, we must heve $y''=0$ and $z''=0$.
And then $\rk(b \mid y, z, w)= \rk(b \mid 0,0,w-\sum_{i=1}^u y_iz_i)$, where $y'=(y_1, \dots y_u),\;z'=(z_1, \dots z_u)$, by row and column operations.
Therefore 
$$E(u,u)= \sum_{y' \in \mathbb{F}^u} \sum_{z' \in \mathbb{F}^u} \overline{\theta}(\sum_{i=1}^u y_iz_i)=\left( \sum_{x \in \mathbb{F}}\sum_{x' \in \mathbb{F}} \overline{\theta}(xx') \right)^u = q^u.
$$
For the last case $r=u+1$, we can use Prop.\ref{Erow} and quickly get $E(u,u+1)=-q^u$.
So we proved (\ref{matE}).

Now we are prepared to use theorems in sec.2.
Firstly by Thm.\ref{main3}, we get
\begin{equation}
 \label{mat-koutai}
\varPhi_{n,m}(v+1, r)=\; q^r\varPhi_{n-1,m-1}(v,r)-q^{r-1}\varPhi_{n-1,m-1}(v,r-1)\qquad (0 \leq v \leq n-1,\; 0\leq r \leq n),
\end{equation}
where interpret $\varPhi_{n-1}(v,-1)=\varPhi_{n-1}(v,n)=0$.
And by Th.\ref{main4}, we get
\begin{equation}
 \label{mat-zensin}
\varPhi_{n,m}(u+1,s)-\varPhi_{n,m}(u,s)=\; -q^{n+m-u-1} \varPhi_{n-1,m-1}(u,s-1) \qquad (0 \leq u \leq n-1,\; 0\leq s \leq n).
\end{equation}
So by setting $f_N(y,x)=\varPhi_{N,N+m-n}(y,x)\;(0 \leq y,x \leq N)$,
(\ref{mat-koutai}) and (\ref{mat-zensin}) imply BPR(\ref{koutai}) and FPR(\ref{zensin}) of the case $a=b=c=1,\; d=q^{m-n},\; t=q$ and $\sigma=\varPhi_{0,m-n}(0,0)=1$.
So we conclude by Prop.\ref{zenkasikitoku}(3),
\begin{eqnarray}
 \label{mat-kidou}
|\mathcal{O}_{n,m}(r)|=\; (-1)^r q^{\binom{r}{2}} (q^m; q^{-1})_r \begin{bmatrix} n \\ r \end{bmatrix} _q,\quad \text{and} \\
\varPhi_{n,m}(s,r)=\; |\mathcal{O}_{n,m}(r)|\, K_s^{\Aff}(r; q^{-m}, n; q).
\end{eqnarray}

Similarly to Cor.\ref{seisitu1}, we can prove properties of $\varPhi_{n,m}(s,r)$:
\begin{col}
\label{seisitu2}
{\rm (1)}\quad $\varPhi_{n,m}(s,r)$ is a polynomial in $q$ over integers.

{\rm (2)}\quad As $q \rightarrow 1$, if $r \leq s$ then
$\displaystyle \varPhi_{n,m}(s,r) \rightarrow (-1)^r \binom{s}{r}$,
otherwise $\displaystyle \varPhi_{n,m}(s,r) \rightarrow 0$.
\end{col}
\noindent
\underline{Remark}. \;(2) implies that in this example also, the reversible formula (\ref{sin-hanten}) is a $q$-analogue of the reversible formula in Prop.\ref{mori}.

Next we derive the generating function as follows:
\begin{pp}
 \label{mat-bo}
Let $t$ be a variable. We have
$$
\sum_{r=0}^n \varPhi_{n,m}(s,r) t^r=\;(t; q)_s \cdot \sum_{u =0}^{n-s} (-1)^u q^{\binom{u}{2} +su} (q^{m-s}; q^{-1})_u \begin{bmatrix} n-s \\ u \end{bmatrix}_q\,t^u        \qquad (0 \leq s \leq n \leq m).
$$
\end{pp}

\begin{proof}
It is proved just like Prop.\ref{diag-bo}.
Here combined diagram is as follows:
\begin{equation}
 \label{mat-comb-d}
\begin{CD}
\mathbb{C}[_{G_{n,m}}\!\backslash \M_{n,m}] @>{\mathcal{F}_{n,m}}>>\mathbb{C}[_{G_{n,m}}\!\backslash (\M_{n,m})\,\hat{}\;] \\
@V{\prod^s(\pi_{*})}V{\qquad\qquad\qquad\quad\;\;\circlearrowright}V @VV{\prod^s (\hat{\pi}^*)}V \\
\mathbb{C}[_{G_{n-s,m-s}}\!\backslash \M_{n-s,m-s}] @>>{\mathcal{F}_{n-s,m-s}}> \mathbb{C}[_{G_{n-s,m-s}}\!\backslash (\M_{n-s,m-s})\,\hat{}\;]
\end{CD}
\end{equation}
Let $\displaystyle \prod^sE$ and $\displaystyle \prod^s\Delta$ be the matrices of $\prod^s(\pi_{*})$ and $\prod^s(\hat{\pi}^*)$ respectively.
Remark here
$E\left(t^i\right)_{i=0}^n=(1-t)\left((qt)^i\right)_{i=0}^{n-1}$
by (\ref{matE}).
Thus inductively, 
\begin{equation}
 \label{matpapa}
\varPhi_{n-s,m-s} (\prod^sE)\left(t^i\right)_{i=0}^n
=\;(t; q)_s\,\varPhi_{n-s,m-s} \left((q^st)^i\right)_{i=0}^{n-s}
=\;(t; q)_s \left( \sum_{u=0}^{n-s} \varPhi_{n-s, m-s}(i,u) q^{su}t^u \right)_{i=0}^{n-s}.
\end{equation}
On the other hand, $\displaystyle (\prod^s\Delta) \varPhi_{n,m} \left(t^i\right)_{i=0}^n$ is just like (\ref{mama}).
Comparing first elements of them and substituting (\ref{mat-kidou}), we get the statement.
\end{proof}
\noindent
We can derive the generating function of Affine $q$-Krawtchouk polynomials \cite{Koe} also as follows:\; Let $T$ be a variable,
\begin{equation}
 \label{affkraw-bo}
\sum_{y=0}^N (a; q)_y \begin{bmatrix} N \\ y \end{bmatrix}_q K_y^{\Aff}(x; a, N; q) \,T^y
=\;(aT; q)_x \cdot
\sum_{u =0}^{N-x} (q^xa; q)_u \begin{bmatrix} N-x \\ u \end{bmatrix}_q\,T^u        \qquad (0 \leq x \leq N,\;a \in \mathbb{C} -\{0,1\}).
\end{equation}
In fact, substituting Prop.\ref{mat-me}, setting $a=q^{-m}$ and changing variable in Prop.\ref{mat-bo},
it becomes an equation of rational functions in $a$. And for $m \geq n$ be arbitrary, we have the equation for $\forall a \in \mathbb{C} -\{0,1\}$.

Next let's derive a multi-orthogonality relation of $\varPhi_{n,m}(s,r)$ or Affine $q$-Krawtchouk polynomials.
\begin{pp}
 \label{mat-cho}
\; Let $k \geq 1$ and $0 \leq r_1,\cdots, r_k,r \leq n$ such that $\displaystyle \sum_{i} r_i \leq r$, then
$$
\sum_{s=0}^n (-1)^s q^{\binom{s}{2}} (q^m; q^{-1})_s \begin{bmatrix} n \\ s \end{bmatrix}_q\, \prod_{i}\varPhi_{n,m}(s,r_i)\varPhi_{n,m}(s,r)
=\; \delta_{\sum_{i}r_i, r} \cdot(-1)^r q^{nm+\binom{r}{2} + \sum_{i<j}r_ir_j}
\frac{(q^m; q^{-1})_r \, (q^n; q^{-1})_r}
{\prod_{i}(q; q)_{r_i}}.
$$
\end{pp}
\begin{proof}
We count the number $N:=\sharp \{ (a_1,\cdots,a_k) \in \mathcal{O}({r_1}) \times \cdots \times \mathcal{O}({r_k}) \mid a_1+ \cdots +a_k \in \mathcal{O}(r) \}.$
When $\sum r_i < r$, obviously $N=0$.
For counting the case $\sum r_i = r$, we use a lemma following:
\begin{lem}
 \label{number1}
Let $0 \leq l,\!k \leq n \leq m$ and $a \in \M_{n,m},\; \rk a=l$. Then
$$
\sharp\{ b \in \M_{n,m} \mid \rk b=k,\; \rk(a+b)=l+k \}=\;q^{2kl}|\mathcal{O}_{n-l,m-l}(k)|.
$$
\end{lem}
\noindent
In fact, it's enough to think when
$a=
\left(\!\!\!\!
\begin{array}{ccc}
 \begin{array}{ccc}
 1&&\\[-10pt]
 & \!\!\!\!\!\ddots&\\[-5pt]
 &&\!\!\!\!\!1
 \end{array}
&&\\[0pt]
&& 
\end{array}
\right)
 \in \M_{n,m}
$, where the number of $1$ is $l$
(For a general case, we can make an adequate correspondence with it).
Then for $b \in \mathcal{O}_{(n,m)}(k)$, the condition that $\rk(a+b)=l+k$ is equivalent to that the 
lower right side $(n-l) \times (m-l)$ sized matrix in $b$ has rank $k$.
It implies the lemma.

\vspace{0.1in}

Using the lemma, taking $a_1,\cdots, a_k$ in order such like $a_i \in \mathcal{O}({r_i}),\; a_1+\cdots +a_i \in \mathcal{O}({r_1+\cdots +r_i})\;(1 \leq i \leq k)$,
we have
$$
N=\; \prod_{i=1}^k q^{2(r_1+\cdots +r_{i-1})r_i}\,|\,\mathcal{O}_{n-r_1-\cdots -r_{i-1},\, m-r_1-\cdots -r_{i-1}}({r_i})\,|.
$$
Substituting (\ref{mat-kidou}) and calculating, we get $\displaystyle N=(-1)^r q^{\binom{r}{2} + \sum_{i<j}r_ir_j}
\frac{(q^m; q^{-1})_r \, (q^n; q^{-1})_r}
{\prod_{i}(q; q)_{r_i}}$.
Then by Prop.\ref{taju} we get the result.
\end{proof}

Again substituting Prop.\ref{mat-me} and setting $a=q^{-m}$, we get
\begin{col}
 \label{kraw-cho}
Let $k \geq 1$ and $0 \leq y_1,\cdots, y_k,y \leq N$ such that $\displaystyle \sum_{i} y_i \leq y$, then
$$
\sum_{x=0}^{N}\,a^{N-x} (a;q)_x
\begin{bmatrix} N \\ x \end{bmatrix} _q\,\prod_i K^{{\rm Aff}}_{y_i}(x;a,N;q) K^{{\rm Aff}}_y(x;a,N;q)
=\; \delta_{\sum_i y_i, y}\; 
q^{\sum_{i < j}y_iy_j}
\frac{a^y (q;q)_y}{\prod_{i}(a;q)_{y_i}(q^N; q^{-1})_{y_i}}.
$$
\end{col}
\noindent
\underline{Remark}.\; When k=1, it becomes the ordinary orthogonality relation of Affine $q$-Krawtchouk polynomial \cite{Koe} as follows:
\begin{equation}\nonumber
\sum_{x=0}^{N}\,a^{N-x} (a;q)_x
\begin{bmatrix} N \\ x \end{bmatrix} _q\, K^{{\rm Aff}}_{y}(x;a,N;q) K^{{\rm Aff}}_{y'}(x;a,N;q)
=\; \delta_{y, y'}\; 
\frac{a^y}{(a;q)_{y}}\begin{bmatrix} N \\ y \end{bmatrix} _q^{-1}.
\end{equation}


\subsection{The pair $(\A_n, \G_n)$}
In this subsection, we assume $q$ is odd.
We consider the action $\rho$ of $\G_n$ on $\A_n$ of all $n \times n$-sized alternating (skew-symmetric) matrices such that
\begin{equation}
\label{act3}
\rho(g)a=\; ga\,{}^t\!g \qquad (g \in \G_n,\; a \in \A_n).
\end{equation}
Remark when regarding $\G_n$ as a subgroup of $\G_n \times \G_n$ by the diagonal embedding $g \mapsto (g,g)$, this action is a sub-action of (\ref{act2}).
First we assume a fact concerning rank of alternating matrices \cite{MW}:

\begin{fact}
\label{jijitu3}
\; The rank of any alternating matrix is even. And any two $n \times n$-sized alternating matrices of the same rank are transitive each other by the action $\rho$ above.
\end{fact}
\noindent
By the fact, orbits of this action are parametrized as follows:
\begin{equation}
_{\G_n}\! \backslash \A_n =\; \{\mathcal{O}(r) \mid 0 \leq r \leq n,\; r \;\text{is even}\}, \qquad \mathcal{O}(r) =\;\mathcal{O}_n(r) = \{a \in \A_n|\; \rk a = r \}.
\end{equation}
$\rho$ is adjoint-free corresponding to the form $\langle \; | \; \rangle$ in (\ref{form}).
Thus orbits of the character group $(\A_n)\,\hat{}\;$ are also parametrized by $\{s \mid 0 \leq s \leq n,\; s \;\text{is even}\}$ by the correspondence (\ref{taiou}).
Let $\mathcal{F}_n: \mathbb{C}[_{\G_{n}}\! \backslash \A_{n}] \rightarrow  \mathbb{C}[_{\G_{n}}\! \backslash (\A_{n})\,\hat{}\;]$ be the group-invariant Fourier transformation,
and $\varPhi_{n}=(\varPhi_{n}(s,r))$ be the canonical matrix of it.
Remark when $n=0$ and $n=1$, we can regard $\A_{n}=0$ and $\varPhi_{n}(0,0)=1$.
Now we see followings:

\begin{pp}
 \label{alt-me}
\; For $0 \leq x,\!y \leq N$,
\begin{eqnarray}\nonumber
\begin{split}
&\varPhi_{2N}(2y,2x)
=\; (-1)^x q^{x(x-1)} (q^{2N-1};q^{-2})_x \begin{bmatrix} N \\ x
\end{bmatrix} _{q^2} \, K^{\mathrm{Aff}}_y(x; q^{-2N+1}, N; q^2),\quad \text{and}\\
&\varPhi_{2N+1}(2y,2x)
=\; (-1)^x q^{x(x-1)} (q^{2N+1};q^{-2})_x \begin{bmatrix} N \\ x
\end{bmatrix} _{q^2} \, K^{\mathrm{Aff}}_y(x; q^{-2N-1}, N; q^2).
\end{split}
\end{eqnarray}
\end{pp}
We  calculate it below.
Let $n \geq 2$.
We use a notation
\begin{equation}
 \label{notation2}
 (b||\, z, -^t\!z, w):=\; \left(\begin{array}{cc}\\[-10pt]
b& z \\[0pt]
-^t\!z& w  \end{array}\right)
\in \A_{n} \qquad (b \in \A_{n-2},\; z \in \M_{n-2, 2},\;w \in \A_2).
\end{equation}
Let $\pi:\A_{n} \rightarrow \A_{n-2}$ be a projection map $(b||\, z, -^t\!z,w) \mapsto b$.
Take an intersection character $\theta_e$ corresponding to $ \displaystyle e=(0||\,0,0,\; \left(\begin{array}{cc}\\[-10pt]
0& 1 \\[0pt]
-1& 0  \end{array}\right)) \in \A_{n}$ and define $\pi_*=\pi_*^e$ and $\hat{\pi}=\hat{\pi}_e$.
Then for $b \in \A_{n-2}$, the rank of ${}^t\!\pi(b) +e$ is bigger by 2 than that of $b$.
So , we get $\tilde{v}=\hat{\pi}(v)=v+2$ as we see $\hat{\pi}:\{v \mid 0 \leq v \leq n-2,\; \text{even}\} \to \{s \mid 0 \leq s \leq n,\; \text{even}\}$ by (\ref{hatpi-para}), the matrix of $\hat{\pi}^*$ on canonical bases such that
\begin{equation}
 \label{altDelta}
\Delta=(\delta_{v+2,s})_{0\leq v \leq n-2,\; 0 \leq s \leq n,\; \text{both even}},
\end{equation}
and following diagram:
\begin{equation}
 \label{alt-d}
\begin{CD}
\mathbb{C}[_{\G_{n}}\!\backslash \A_{n}] @>{\mathcal{F}_{n}}>>\mathbb{C}[_{\G_{n}}\!\backslash (\A_{n})\,\hat{}\;] \\
@V{\pi_{*}}V{\qquad\qquad\quad\;\;\circlearrowright}V  @VV{\hat{\pi}^*}V \\
\mathbb{C}[_{\G_{n-2}}\!\backslash \A_{n-2}] @>>{\mathcal{F}_{n-2}}> \mathbb{C}[_{\G_{n-2}}\!\backslash (\A_{n-2})\,\hat{}\;]
\end{CD}
\end{equation}
Let $E=(E(u,r))_{u,r}$ be the matrix of $\pi_*$ on canonical bases.
We have
\begin{equation}
 \label{altE}
E(u,r)=\; \begin{cases} q^u & (r=u)\\ -q^u & (r=u+2) \\ 0 & (otherwise) \end{cases}
\qquad\quad (0\leq u \leq n-2,\;0 \leq r \leq n,\; \text{both even}).
\end{equation}
It can be calculated as (\ref{matE}) in the example of $\M_n$. 
Let's use
$\displaystyle
b=
\left(\!\!\!\!
\begin{array}{cc}&\\[-10pt]
\begin{array}{ccccc}
0&\!\!\!\!\!1&&&\\[-3pt]
-1&\!\!\!\!\!0&&&\\[-5pt]
&& \!\!\!\!\ddots&&\\[-5pt]
&&&\!\!\!\!\!0&\!\!\!\!\!1\\[-3pt]
&&&\!\!\!\!-1&\!\!\!\!0
\end{array}
&\\
&\!\!\!\!\mbox{\smash{\LARGE{0}}}
\end{array}
\right)
\in \mathcal{O}_{n-2}(u)
$ as a representative of the orbit.
Then,
$\displaystyle E(u,r)= \pi_*\chi_r(b)
= \sum_{z \in \M_{n-2,2}} \sum_{w \in \mathbb{F}} \overline{\theta(2w)}\chi_r\left(b||\, z,\,-^t\!z,\; \left(\begin{array}{cc}\\[-10pt]
0& w \\[0pt]
-w& 0  \end{array}\right)\right)$.
So we find $E(u,r)=0$ for $r \neq u,\,u+2,\;u+4$ at once.
And since the condition $\rk\left(b||\, z,\,-^t\!z,\; \left(\begin{array}{cc}\\[-10pt]
0& w \\[0pt]
-w& 0  \end{array}\right)\right)=u+4$ doesn't depend on $w \in \mathbb{F}$, we have $E(u,u+4)=0$.
Next when $r=u$, since 
\\
$\displaystyle \rk\left(b||\, (z',0), -\,^t\!(z',0),\; \left(\begin{array}{cc}\\[-10pt]
0& w \\[0pt] -w& 0  \end{array}\right)\right)
=\; \rk \left(b|| \, 0,0,\; w-\sum_{i:\odd} (z_{i,1}z_{i+1,2}-z_{i,2}z_{i+1,1})\left(\begin{array}{cc}\\[-10pt]
0& 1 \\[0pt] -1& 0  \end{array}\right) \right)$ where $z'=(z_{i,j}) \in \M_{u,2}$,
we have
$\displaystyle
E(u,u)= \sum_{z' \in \M_{u,2}} \overline{\theta}\left(2\sum_{i:\odd} (z_{i,1}z_{i+1,2}-z_{i,2}z_{i+1,1})\right)
=\; \left( \sum_{x \in \mathbb{F}}\sum_{x' \in \mathbb{F}} \overline{\theta}(2xx') \right)^u 
\\
= q^u.
$
At the end $E(u,u+2)=-q^u$ by Prop.\ref{Erow}, and we got (\ref{altE}).

Now from the diagram (\ref{alt-d}) (i.e. from Thm.\ref{main3}, \ref{main4}), we get
\begin{equation}
 \label{alt-koutai}
\varPhi_{n}(v+2, r)=\; q^r\varPhi_{n-2}(v,r)-q^{r-2}\varPhi_{n-2}(v,r-2)\qquad (0 \leq v \leq n-2,\; 0\leq r \leq n,\; \text{both enen})
\end{equation}
\begin{equation}
 \label{alt-zensin}
\text{and}\qquad \varPhi_{n}(u+2,s)-\varPhi_{n}(u,s)=\; -q^{2n-u-3} \varPhi_{n-2}(u,s-2) \qquad (0 \leq u \leq n-2,\; 0\leq s \leq n,\; \text{both enen}),
\end{equation}
where interpret $\varPhi_{n-2}(v,-2)=\varPhi_{n-2}(v,n)=0$.
Now we separate the family $\{\varPhi_n\}_{n \geq 0}$ of canonical matrices according to whether $n$ is even or odd. That is, we define two families $\{f_{N}\}_{N \geq 0}$ and $\{g_{N}\}_{N \geq 0}$ by
\begin{equation}
f_{N}(y,x)= \varPhi_{2N}(2y,2x),\quad g_N(y,x)= \varPhi_{2N+1}(2y,2x)\qquad  (0 \leq x,\!y \leq N).
\end{equation}
Then by (\ref{alt-koutai}) and (\ref{alt-zensin}),
$\{f_N\}$ satisfies BPR(\ref{koutai}) and FPR(\ref{zensin}) of the case $a=b=c=1,\; d=q^{-1},\; t=q^2$ and $\sigma=\varPhi_{0}(0,0)=1$.
And so is $\{g_N\}$ of the case $a=b=c=1,\; d=q,\; t=q^2$ and $\sigma=\varPhi_{1}(0,0)=1$.
So we conclude by Prop.\ref{zenkasikitoku}(3),
\begin{equation}
\label{alt-kidou}
|\mathcal{O}_n(2x)|
=\;(-1)^x q^{x(x-1)} \frac{(q^n;q^{-1})_{2x}}{(q^2;q^2)_x}\qquad(0 \leq 2x \leq n),
\end{equation}
\begin{eqnarray}
\begin{split}
&\varPhi_{2N}(2y,2x)
=\; |\mathcal{O}_{2N}(2x)| \, K^{\mathrm{Aff}}_y(x; q^{-2N+1}, N; q^2), \quad \text{and}\\
&\varPhi_{2N+1}(2y,2x)
=\; |\mathcal{O}_{2N+1}(2x)| \, K^{\Aff}_y(x; q^{-2N-1}, N; q^2),
\end{split}
\end{eqnarray}
which proved Prop.\ref{alt-me}.
\vspace{0.1in}

Similarly to Cor.\ref{seisitu1}, \ref{seisitu2}, we get again properties of $\varPhi_{n}(s,r)$ as follows:
\begin{col}
\label{seisitu3}
{\rm (1)}\quad $\varPhi_{n}(s,r)$ is a polynomial in $q$ over integers.

{\rm (2)}\quad As $q \rightarrow 1$, if $x \leq y$ then
$\displaystyle \varPhi_n(2y,2x) \rightarrow (-1)^x \binom{y}{x}$,
otherwise $\displaystyle \varPhi_n(2y,2x) \rightarrow 0$.
\end{col}
\noindent
So it also gives a $q$-analogue of the reversible formula in Prop.\ref{mori}.


\section{Group-invariant Fourier transformations on Symmetric matrices over a finite field}

Throughout this section also $\mathbb{F}$ is a finite field, and the order $q$ is odd.
Here we work with a vector space $\S_n$ of $n \times n$-sized symmetric matrices over $\mathbb{F}$.
We consider two group actions on $\S_n$.
The first one is $\G_n$-action $\rho$ as follows:
\begin{equation}
 \label{s-act1}
\rho(g)a=\; ga\,{}^t\!g \qquad (\,g \in \G_n,\;a \in \S_n).
\end{equation}
The second one is $\mathbb{F}^\times \times \G_n$-action $\rho$ as follows:
\begin{equation}
 \label{s-act2}
\rho(c, g)a=\; cga\,{}^t\!g \qquad (\,c \in \mathbb{F}^\times,\; g \in \G_n,\;a \in \S_n).
\end{equation}
We consider group-invariant Fourier transformations
\begin{eqnarray}
 \label{s-ft1}
&&\mathcal{F}=\mathcal{F}_n:\mathbb{C}[\;_{\G_n} \!\backslash \S_n] \rightarrow \mathbb{C}[\;_{\G_n} \!\backslash (\S_n)\,\hat{}\,] \qquad \text{and}\\
 \label{s-ft2}
&&\mathcal{F}=\mathcal{F}_n:\mathbb{C}[\;_{\mathbb{F}^\times \times \G_n} \!\backslash \S_n] \rightarrow \mathbb{C}[\;_{\mathbb{F}^\times \times \G_n} \!\backslash (\S_n)\,\hat{}\,]
\end{eqnarray}
concerning each actions.
Remark that (\ref{s-act1}) is regarded as a sub-action of (\ref{s-act2}), and (\ref{s-ft2}) is regarded as a restriction of (\ref{s-ft1}).
Here we solve matrix elements of $\mathcal{F}_n$ under an adequate base change from canonical bases.

Firstly we prepare some notations used in this section.
Define sign of $a \in \mathbb{F} \;(a \neq 0)$ by
\begin{equation}
 \label{sgn1}
\sgn(a)=\begin{cases} 1 &  \text{if $a$ is a square element,}\\
 -1 &  \text{if $a$ is non-square element}. 
\end{cases}
\end{equation}
Fix a non-trivial additive character $\theta \in \hat{\mathbb{F}}$.
And define a constant $\gamma$ by
\begin{equation}
 \label{gamma}
\gamma =\; \sum_{a \in \mathbb{F}-\{0\}} \sgn(a) \overline{\theta(a)},
\end{equation}
which is called {\bf Gauss sum} corresponding to $\theta$. 
Fix a non-square element $\delta \in \mathbb{F}$.
Let
\begin{equation}
  \epsilon=\;  \begin{cases}
    1 & \mathrm{if} \quad q \equiv 1(\rm{mod}\; 4), \\
    -1 & \mathrm{if} \quad q \equiv 3(\rm{mod}\; 4).
  \end{cases}
\end{equation}
Remark $\epsilon = \sgn(-1)$. In fact, if $q$ is a prime, it is a famous statement of the quadratic residue.
It is not difficult to extend to general prime powers $q$ (We omit the details).

Let {\rm diag}$(a_1,\dots,a_n)$ denote $n \times n$-sized diagonal matrix with diagonal elements $a_1,\dots,a_n \in \mathbb{F}$.


\subsection{Overview, canonical bases and changed bases}
\label{5-1}

As for orbits of the $\G_n$-action,
the following fact is known\cite{MW}:
\begin{fact}
\label{s-kidou}
\; Any non-zero symmetric matrix is transitive to only one of diagonal matrices \\
$\diag(1, \dots 1,0, \dots ,0)$ and 
$\diag(1,\dots,1,\delta,0,\dots,0)$ by $\G_n$-action \rm{(\ref{s-act1})}.
\end{fact} 
\noindent
By the fact, we can define sign of $a \in \S_n\;(a \neq 0)$ by
\begin{equation}
 \label{sgn}
\sgn(a)=\begin{cases} 1 &  \text{if $a$ is transitive to $\diag(1, \dots 1,0, \dots ,0)$}
\\ -1 &  \text{if $a$ is transitive to $\diag(1, \dots 1,\delta,0, \dots ,0)$}. 
\end{cases}
\end{equation}
It is consistent with (\ref{sgn1}) for $n=1$.
For convenience, let $\sgn(0)=1$.
We can easily verify properties
\begin{equation}
 \label{sgn-seki}
\sgn\left(\begin{array}{cc}\\[-12pt]
a& 0 \\[0pt]
0& b  \end{array}\right)=\; \sgn(a)\cdot\sgn(b) \qquad(a \in \S_k,\; b \in \S_{l},\; k,l \geq 1),
\end{equation}
\begin{equation}
 \label{del-sgn}
\sgn(\delta a)=\begin{cases} \sgn(a) & \text{if $\rk(a)$ is even (including $0$)}\\
-\sgn(a) & \text{if $\rk(a)$ is odd} \end{cases}
\qquad (a \in \S_n).
\end{equation}
Then $\G_n$-orbits of $\S_n$ are as follows:
The zero matrix forms an orbit $\mathcal{O}(0)$.
Besides it, any orbits are characterized by rank and sign of the elements.
So for $r \geq 1$ let
$\mathcal{O}(r^+)=\{ a \in \S_n \mid \rk(a)=r,\; \sgn(a)=1\}$ and
$\mathcal{O}(r^-)=\{ a \in \S_n \mid \rk(a)=r,\; \sgn(a)=-1\}$,
then
\begin{equation}
_{\G_{n}}\! \backslash \S_{n} =\; \{\mathcal{O}(\lambda) \mid \lambda \in \Lambda \}, \qquad \Lambda =\; \Lambda_n=\; \{0\} \sqcup \{r^{\pm}|\;1 \leq r \leq n\}.
\end{equation}
Since this action is adjoint-free concerning the bilinear form  $\langle \; | \; \rangle$ in (\ref{form}), we can also parametrize $_{\G_n}\!\backslash (\S_n)\,\hat{}$ by $\Lambda$, and define canonical matrix $\varPhi=(\varPhi(\mu, \lambda))_{\mu, \lambda \in \Lambda}$ of $\mathcal{F}$ in (\ref{s-ft1}).
Easy observations induce some properties of $\varPhi$ as follows:

\begin{pp}\label{s-cme}\;Let $1 \leq s,\!r \leq n$ and assume all double-signs below correspond.\\
{\rm(1)}\; For odd $s$ and odd $r$,\;\;$\displaystyle\varPhi(s^+,r^{\pm})=\varPhi(s^-,r^{\mp}).$\quad
{\rm(2)}\; For odd $s$ and even $r$,\;\;$\displaystyle\varPhi(s^+,r^{\pm})=\;\varPhi(s^-,r^{\pm}).$\\
{\rm(3)}\; For even $s$ and odd $r$,\;\;$\displaystyle\varPhi(s^{\pm},r^+)=\;\varPhi(s^{\pm},r^-),$\quad and {\rm(4)}\; For odd $r$,\;\;$\displaystyle\varPhi(0,r^+)=\;\varPhi(0,r^-).$
\end{pp}
\begin{proof}\; For (1), let $a \in \mathcal{O}(s^+)$.
Then $\delta a \in \mathcal{O}(s^-)$ since $s$ is odd.
And when $b$ runs over $\mathcal{O}(r^+)$, $\delta^{-1}b$ runs over $\mathcal{O}(r^-)$ because $r$ is odd.
Thus by (\ref{me}),
$\displaystyle
\varPhi(s^+,r^+)=\; \sum_{b \in \mathcal{O}(r^+)} \theta(\langle a| b\rangle)=\sum_{b \in \mathcal{O}(r^+)} \theta\left(\langle \delta a | \delta^{-1} b \rangle \right)=\varPhi(s^-,r^-).
$
The others are likewise.
\end{proof}

As for orbits of the $\mathbb{F}^\times \times \G_n$-action, we notice followings by (\ref{del-sgn}):
If $r$ is odd, elements in $\mathcal{O}(r^+)$ and $\mathcal{O}(r^-)$ are transitive each other, and if $r$ is even, elements in $\mathcal{O}(r^+)$ and $\mathcal{O}(r^-)$ are not transitive each other by the $\mathbb{F}^\times \times \G_n$-action.
Thus for odd numbers $r$, let $\mathcal{O}(r)= \mathcal{O}(r^+) \sqcup \mathcal{O}(r^-)=\{a \in \S_n \mid \rk(a)=r\}$.
Then
\begin{equation}
 \label{lambda-dash}
_{\mathbb{F}^\times \times \G_{n}}\! \backslash \S_{n} =\; \{\mathcal{O}(\lambda) \mid \lambda \in \Lambda' \}, \quad \Lambda' =\; \Lambda_n'=\; \{0\} \sqcup \{r \mid 1 \leq r \leq n,\; \text{odd}\} \sqcup \{r^{\pm} \mid 1 \leq r \leq n,\; \text{even}\}.
\end{equation}

Next we introduce a basis $\mathcal{B}$ of $\mathbb{C}[_{\G_{n}}\! \backslash \S_{n}]$,
which we use mainly in our exploration rather than the canonical basis $\{\chi_{\lambda} \mid \lambda \in \Lambda \}$.
For any $0 \leq r \leq n$, let $\chi_r$ be the characteristic function for matrices of the rank $r$ on $\S_n$.
And for $1 \leq r \leq n$ define $^{\sgn}\chi_r$ by
\begin{equation}
^{\sgn}\chi_r(a)=\; \sgn(a)\cdot \chi_{r}(a)=\begin{cases} \pm 1 & \text{if}\;\;\; a \in \mathcal{O}(r^{\pm})\; ,\; \text{double-sign corresponds}  \\ 0 & \text{otherwise}
\end{cases}\qquad\quad(a \in \S_n).
\end{equation}
Then $\chi_r=\chi_{r^+} + \chi_{r^-}$ and 
$^{\sgn}\chi_r=\chi_{r^+} - \chi_{r^-}$ hold, and
$\displaystyle
\mathcal{B}=\mathcal{B}_n
=\{ \chi_r \mid 0 \leq r \leq n \} \sqcup \{ ^{\sgn}\chi_r \mid 1 \leq r \leq n \}
$
consists of a basis of $\mathbb{C}[_{\G_{n}}\! \backslash \S_{n}]$.
At the same time, we change bases of $\mathbb{C}[\;_{\G_n} \!\backslash (\S_n)\,\hat{}\,]$ from
canonical basis $\{\chi_{\lambda} \mid \lambda \in \Lambda \}$
to a new basis
$\displaystyle
\hat{\mathcal{B}}
=\{ \chi_s \mid 0 \leq s \leq n \} \sqcup \{ ^{\sgn}\chi_s \mid 1 \leq s \leq n \}
$
by the same relation 
$\chi_s=\chi_{s^+} + \chi_{s^-}$ and 
$^{\sgn}\chi_s=\chi_{s^+} - \chi_{s^-}$ as above.
Subsets
$\displaystyle
\mathcal{B}'
= \{ \chi_r \mid 0 \leq r \leq n \} \sqcup \{ ^{\sgn}\chi_r \mid 1 \leq r \leq n,\; \text{even} \} \subset \mathcal{B}
$
and
$\displaystyle
\hat{\mathcal{B}}'
=\{ \chi_s \mid 0 \leq s \leq n \} \sqcup \{ ^{\sgn}\chi_s \mid 1 \leq s \leq n,\; \text{even} \} \subset \hat{\mathcal{B}}
$
span subspaces $\mathbb{C}[\, _{\mathbb{F}^\times \times \G_{n}}\! \backslash \S_{n}]$ and $\mathbb{C}[\, _{\mathbb{F}^\times \times \G_{n}}\! \backslash (\S_{n})\,\hat{}\,]$
respectively.
Fourier transformation $\mathcal{F}$ (\ref{s-ft1}) corresponds not only these subspaces by (\ref{s-ft2}), but also their complement spaces according to these bases. That is, 

\begin{pp}
 \label{hokukan}
\; Let $U \subset \mathbb{C}[\, _{\G_{n}}\! \backslash \S_{n}]$ be a subspace spanned by $\mathcal{A}= \{ ^{\sgn}\chi_r \mid 1 \leq r \leq n,\;\text{odd}\; \} \subset \mathcal{B}$.
Likewise let $\hat{U} = \mathrm{Span} (\hat{\mathcal{A}}),\;\;\hat{\mathcal{A}}= \{ ^{\sgn}\chi_s \mid 1 \leq s \leq n,\;\text{odd}\; \} \subset \hat{\mathcal{B}}$. Then $\mathcal{F}(U)=\hat{U}$.
\end{pp}
\begin{proof}
For odd $r$, using Prop.\ref{s-cme},\;
$\displaystyle
\mathcal{F}(^{\sgn}\chi_r)
=\; \mathcal{F}(\chi_{r^+} - \chi_{r^-})\\
=\; \left( \varPhi(0, r^+) - \varPhi(0,r^-) \right) \chi_0
+ \sum_{s=1}^n \left\{ \left( \varPhi(s^+, r^+) - \varPhi(s^+,r^-) \right) \chi_{s^+}
+ \left( \varPhi(s^-, r^+) - \varPhi(s^-,r^-) \right) \chi_{s^-}
\right\}\\
= \frac{1}{2} \sum_{\text{odd}\; s} \left( \varPhi(s^+, r^+) - \varPhi(s^+,r^-) - \varPhi(s^-, r^+) + \varPhi(s^-,r^-)  \right)\, ^{\sgn}\chi_s.
$
\end{proof}

Now define the matrix of
$\mathcal{F}$ in (\ref{s-ft1})
on the bases $\mathcal{B}$ and $\hat{\mathcal{B}}$ by
\begin{equation}
 \label{Psi}
\Psi =\Psi_n=\left(\begin{array}{cc}\\[-12pt]
\Psi^{(1)}(s,r)& \Psi^{(2)}(s,r) \\[0pt]
\Psi^{(3)}(s,r)& \Psi^{(4)}(s,r)  \end{array}\right),
\end{equation}
where $\Psi^{(1)}$ is on $0 \leq s,\!r \leq n$, 
$\Psi^{(2)}$ is on $0 \leq s \leq n$ and $1 \leq r \leq n$,
$\Psi^{(3)}$ is on $1 \leq s \leq n$ and $0 \leq r \leq n$,
$\Psi^{(4)}$ is on $1 \leq s,\!r \leq n$.
It means
\begin{eqnarray}
 \label{f-chi}
&&\mathcal{F}(\chi_r)=\; \sum_{s=0}^n \Psi^{(1)}(s,r)\chi_s+\sum_{s=1}^n \Psi^{(3)}(s,r)\,^{\sgn}\chi_s \qquad (0 \leq r \leq n)\quad \text{and}\\
 \label{f-schi}
&&\mathcal{F}(^{\sgn}\chi_r)=\; \sum_{s=0}^n \Psi^{(2)}(s,r)\chi_s+\sum_{s=1}^n \Psi^{(4)}(s,r)\,^{\sgn}\chi_s \qquad (1 \leq r \leq n).
\end{eqnarray}
\noindent
By definition, the relations between $\varPhi$ and $\Psi$ are as follows:
\begin{eqnarray}
 \label{phi-psi}
&&\varPhi(0,r^{\pm})=\; \frac{1}{2}\left( \1(0,r) \pm \2(0,r) \right)\quad(1 \leq r \leq n,\; \text{double-sign corresponds}),\\
&&\varPhi(s^\alpha,r^\beta)=\; \frac{1}{2}\left( \1(s,r) + \beta \2(s,r) + \alpha \3(s,r) + \alpha \beta \4(s,r) \right)
\\ \nonumber
&&\qquad\qquad\qquad\qquad\qquad\qquad\qquad\qquad\qquad\qquad
(1 \leq s,r \leq n,\;\alpha,\beta \;\text{are signs}).
\end{eqnarray}
Since $\mathcal{F}(\chi_0)=1$, easily we have
\begin{equation}
 \label{psi-s-0}
\Psi^{(1)}(s,0)=1\;\;(0 \leq s \leq n),\quad \Psi^{(3)}(s,0)=0\;\;(1 \leq s \leq n).
\end{equation}
Remark we can write $\Psi$ for the matrix of $\mathcal{F}$ in (\ref{s-ft2}) on the bases $\mathcal{B}'$ and $\hat{\mathcal{B}}'$ also.
In fact it is the same matrix $\Psi$ in (\ref{Psi}) just removed rows and columns related to $\mathcal{A}$ and $\hat{\mathcal{A}}$ in Prop.\ref{hokukan}.
See some properties of $\Psi$ as follows:
\begin{pp}
 \label{psi-zero}
\;{\rm(1)}\; For any $s$ and odd $r$,\;\;$\displaystyle\Psi^{(2)}(s,r)=0$.
\quad{\rm(2)}\; For odd $s$ and any $r$, \;\;$\displaystyle\Psi^{(3)}(s, r)=0$.\\
\quad{\rm(3)}\; For odd $s$ and even $r$, or even $s$ and odd $r$,\;\;$\displaystyle\Psi^{(4)}(s,r)=0$.
\end{pp}
\begin{proof}
\; For (1), since $^{\sgn}\chi_r \in U$ in Prop.\ref{hokukan}, $\mathcal{F}(^{\sgn}\chi_r)$ is written in a linear combination on $\hat{\mathcal{A}}$. Thus (\ref{f-schi}) implies.
Others are likewise.
\end{proof}

\subsection{Commutative diagram for the pairs $(\S_n, \G_n)$ and $(\S_{n-1}, \G_{n-1})$}

Here we see a relation between 
$
\mathcal{F}_n:\mathbb{C}[\,_{\G_n} \!\backslash \S_n] \rightarrow \mathbb{C}[\,_{\G_n} \!\backslash (\S_n)\,\hat{}\,]$
and
$
\mathcal{F}_{n-1}:\mathbb{C}[\,_{\G_{n-1}} \!\backslash \S_{n-1}] \rightarrow \mathbb{C}[\,_{\G_{n-1}} \!\backslash (\S_{n-1})\,\hat{}\,]
$
using statements in sec.2,
where $n \geq 1$ and we regard $\S_0=0$.
And we find some relations concerning the matrix $\Psi=\Psi_n$(\ref{Psi}).

Let $\pi: \S_n \rightarrow \S_{n-1},\; (b\mid z, \,^{t}\!z, w) \mapsto b$, where the notation $(b \mid z, \,^{t}\!z, w)\;(b \in \S_{n-1},\; z \in \mathbb{F}^{n-1},\; w \in \mathbb{F})$ is in (\ref{notation1}).
We take an intersection character corresponind to $e=(0\mid 0,0,1)$ and define $\pi_*=\pi_*^e$ and $\hat{\pi}=\hat{\pi}_e$.
Then since $^{t}\!\pi(b)+e=(b\mid 0,0,1)$, ($\natural 3'$) holds and
\begin{equation}
\label{s-hatpi1}
\hat{\pi}(0)=1^+,\quad
\hat{\pi}(v^{\pm})=(v+1)^{\pm}\qquad (1 \leq v \leq n-1,\; \text{double-sign corresponds})
\end{equation}
by (\ref{sgn-seki}), where $\hat{\pi}: \Lambda_{n-1} \rightarrow \Lambda_n$(\ref{hatpi-para}).
And we get a commutative diagram (\ref{d2}) as follows:
\begin{equation}
 \label{s-d1}
\begin{CD}
\mathbb{C}[_{\G_{n}}\!\backslash \S_{n}] @>{\mathcal{F}_{n}}>>\mathbb{C}[_{\G_{n}}\!\backslash (\S_{n})\,\hat{}\;] \\
@V{\pi_{*}}V{\qquad\qquad\quad\;\;\circlearrowright}V  @VV{\hat{\pi}^*}V \\
\mathbb{C}[_{\G_{n-1}}\!\backslash \S_{n-1}] @>>{\mathcal{F}_{n-1}}> \mathbb{C}[_{\G_{n-1}}\!\backslash (\S_{n-1})\,\hat{}\;]
\end{CD}
\end{equation}

Let $\displaystyle \Delta=\left(\begin{array}{cc}\\[-12pt]
\Delta^{(1)}(v,s)& \Delta^{(2)}(v,s) \\[0pt]
\Delta^{(3)}(v,s)& \Delta^{(4)}(v,s)  \end{array}\right)$
be the matrix of $\hat{\pi}^*$ on the bases $\hat{\mathcal{B}}_n$ and $\hat{\mathcal{B}}_{n-1}$.
That is,
$\displaystyle
\hat{\pi}^* (\chi_s)=\; \sum_{v=0}^{n-1} \Delta^{(1)}(v,s)\chi_v+\sum_{v=1}^{n-1} \Delta^{(3)}(v,s)\,^{\sgn}\chi_v\;\;(0 \leq s \leq n),\;
$
and
$\displaystyle
\hat{\pi}^*\,(^{\sgn}\chi_s)=\; \sum_{v=0}^{n-1} \Delta^{(2)}(v,s)\chi_v+\sum_{v=1}^{n-1} \Delta^{(4)}(v,s)\,^{\sgn}\chi_v\;\;(1 \leq s \leq n)
$.
Then we have
\begin{eqnarray}
 \label{s-Delta1}
\begin{split}
\Delta^{(1)}(v,s)=\;\delta_{v+1,s},\quad
&\Delta^{(2)}(v,s)=\;\delta_{v,0} \cdot\delta_{s,1},\\
\Delta^{(3)}(v,s)=\;0,\quad
&\Delta^{(4)}(v,s)=\;\delta_{v+1,s}.
\end{split}
\end{eqnarray}
In fact,
by calculations using (\ref{s-hatpi1}),
$\displaystyle
\Delta^{(1)}(0,s)
=\; \hat{\pi}^* \chi_s (0)
=\; \chi_s(1^+)
=\; \delta_{1,s},\;\;
$
and
$\displaystyle
\Delta^{(1)}(v,s)
=\; \frac{1}{2}\left( \hat{\pi}^* \chi_s (v^+) + \hat{\pi}^* \chi_s (v^-) \right)
=\; \frac{1}{2}\left( \chi_s(v+1^+) + \chi_s(v+1^-) \right)
=\; \delta_{v+1,s}\;(v \geq 1).
$
Likewise
$\displaystyle
\Delta^{(2)}(0,s)
=\; \hat{\pi}^* \,^{\sgn}\chi_s (0)
=\; ^{\sgn}\chi_s(1^+)
=\; \delta_{s,1},\;\;
$
and
$\displaystyle
\Delta^{(2)}(v,s) 
=\; \frac{1}{2}\left( \hat{\pi}^* \,^{\sgn}\chi_s (v^+) + \hat{\pi}^* \,^{\sgn}\chi_s (v^-) \right)
\\=\; \frac{1}{2}\left( ^{\sgn}\chi_s(v+1^+) + ^{\sgn}\chi_s(v+1^-) \right)
=0 \;(v \geq 1).
$
And so on.

On the other hand, 
let $\displaystyle E=\left(\begin{array}{cc}\\[-12pt]
E^{(1)}(u,r)& E^{(2)}(u,r) \\[0pt]
E^{(3)}(u,r)& E^{(4)}(u,r)  \end{array}\right)$
be the matrix of $\pi_*$ on the bases $\mathcal{B}_n$ and $\mathcal{B}_{n-1}$.
That is,
$\displaystyle
{\pi}_* (\chi_r)=\; \sum_{u=0}^{n-1} E^{(1)}(u,r)\chi_u+\sum_{u=1}^{n-1} E^{(3)}(u,r)\,^{\sgn}\chi_u\;\;(0 \leq r \leq n),\;
$
and
$\displaystyle
{\pi}_*\,(^{\sgn}\chi_r)=\; \sum_{u=0}^{n-1} E^{(2)}(u,r)\chi_u+\sum_{u=1}^{n-1} E^{(4)}(u,r)\,^{\sgn}\chi_u\;\;(1 \leq r \leq n)
$.
Then we have
\begin{eqnarray}
 \label{s-E1}
\begin{split}
&E^{(1)}(0,0)=1,\quad E^{(1)}(0,1)= -1,\\
&E^{(2)}(u,u)=\gamma^u,\quad E^{(2)}(u,u+1)=\gamma^{u+1} \quad(0 \leq u \leq n-1),\\
&E^{(3)}(u,u)=\gamma^u,\quad E^{(3)}(u,u+1)= -\gamma^{u} \quad(1 \leq u \leq n-1),\\
\text{and}\;&\; \text{except for the above,}\;\; E^{(i)}(u,r)=0\quad(i=1,2,3,4).
\end{split}
\end{eqnarray}
We calculate it below.
Assume $u \geq 1$.
Let's use $b^+= \diag(1, \cdots, 1,0,\cdots,0) \in \mathcal{O}_{n-1}(u^+)$ and $b^-= \diag(1, \cdots, 1,\delta^{-1},0,\cdots,0) \in \mathcal{O}_{n-1}(u^-)$ as representatives of orbits.
When writing $z=\binom{z'}{z''},\; z' \in \bbf^u,\; z'' \in \bbf^{n-u-1}$ for $z \in \bbf^{n-1}$ and
if $z'' \neq 0$,
we have
$\displaystyle
\;\chi_r \left(  b^{\pm}\,|\, z, \,^t\!z, w \right)
=\;
\chi_r \left(  b^{\pm}\,|\, z, \,^t\!z, 0 \right)
$
by row and column operations.
Since it doesn't depend on $w \in \bbf$, then
$\displaystyle
\sum_{w \in \bbf} \overline{\theta}(w) \chi_r \left(  b^{\pm}\,|\, z, \,^t\!z, w \right)
=0.
$
Therefore,
\begin{eqnarray}
 \begin{split}\nonumber
\pi_*\chi_r&(u^{+})=\; \sum_{z \in \bbf^{n-1}}\; \sum_{w \in \mathbb{F}} \overline{\theta}(w) \chi_r \left(  b^{+}\,|\, z, \,^t\!z, w  \right)
=\;\sum_{z' \in \bbf^u}\; \sum_{w \in \mathbb{F}} \overline{\theta}(w) \chi_r \left(  b^{+}\,||\, z', \,^t\!z', w \right)
\\
&=\; \;\sum_{z' \in \bbf^u}\; \sum_{w \in \mathbb{F}} \overline{\theta}(w) \chi_{r-u}\left(\!
w-\sum_{i=1}^u z_{i}^2 \right)
=\; \;\left(\sum_{x \in \bbf} \overline{\theta}(x^2)\right)^u
\; \sum_{y \in \mathbb{F}} \overline{\theta}\left(y\right) \chi_{r-u}\left(y\right)
\quad
=\; \begin{cases} \gamma^u & (r=u) \\ -\gamma^u & (r=u+1)\\ 0 & (\text{otherwise}) \end{cases},
 \end{split}
\end{eqnarray}
where we used $\sum_{x \in \mathbb{F}} \overline{\theta}(x^2)=\gamma$.
By calculating similarly, we have
$\pi_*\chi_r(u^-)= -\pi_*\chi_r(u^+)$.
Therefore $E^{(1)}(u,r)=\frac{1}{2}\left( \pi_*\chi_r(u^+) + \pi_*\chi_r(u^-) \right)=0\;(u \geq 1)$,
and $E^{(3)}(u,r)=\frac{1}{2}\left( \pi_*\chi_r(u^+) - \pi_*\chi_r(u^-) \right)=\pi_*\chi_r(u^+)$.
Similarly, 
\begin{eqnarray}
 \begin{split}\nonumber
\pi_* \,^{\sgn} \chi_r&(u^{+})
=\; \;\left(\sum_{x \in \bbf} \overline{\theta}(x^2)\right)^u
\; \sum_{y \in \mathbb{F}} \overline{\theta}\left(y\right)\,  ^{\sgn}\chi_{r-u}\left(y\right)
\quad
=\; \begin{cases} \gamma^u & (r=u) \\ \gamma^{u+1} & (r=u+1)\\ 0 & (\text{otherwise}) \end{cases},
 \end{split}
\end{eqnarray}
and $\pi_* \,^{\sgn} \chi_r(u^{-})
=\pi_* \,^{\sgn} \chi_r(u^{+})$.
Therefore $E^{(2)}(u,r)=\frac{1}{2}\left( \pi_*\,^{\sgn}\chi_r(u^+) + \pi_*\,^{\sgn}\chi_r(u^-) \right)\\ =\pi_* \,^{\sgn} \chi_r(u^{+})\;(u \geq 1)$,
and $E^{(4)}(u,r)=\frac{1}{2}\left( \pi_* \,^{\sgn} \chi_r(u^{+}) - \pi_* \,^{\sgn}\chi_r(u^-) \right)=0$.
We omit the case of $u=0$, so we finished calculations of (\ref{s-E1}).

Now using the diagram (\ref{s-d1}), matrices $\Delta$ (\ref{s-Delta1}) and $E$ (\ref{s-E1}),
we can induce some relations between matrix $\Psi_n$ and $\Psi_{n-1}$.
Firstly for $1 \leq r \leq n$, an eqution
$\hat{\pi}^* \mathcal{F}_n (\chi_r)= \mathcal{F}_{n-1} \pi_* (\chi_r)$
induces
\begin{eqnarray}
 \label{jun-1}
&&\Psi^{(1)}_n(v+1, r)=\; \gamma^r \Psi^{(2)}_{n-1}(v,r) - \gamma^{r-1} \Psi^{(2)}_{n-1}(v, r-1)\qquad(0 \leq v \leq n-1), \quad\text{and}\\
 \label{jun-2}
&&\Psi^{(3)}_n(v+1, r)=\; \gamma^r \Psi^{(4)}_{n-1}(v,r) - \gamma^{r-1} \Psi^{(4)}_{n-1}(v, r-1)\qquad(1 \leq v \leq n-1),
\end{eqnarray}
where consider $\Psi^{(2)}_{n-1}(v, 0)=1$ and $\Psi^{(4)}_{n-1}(v, 0)= \Psi^{(2)}_{n-1}(v, n)= \Psi^{(4)}_{n-1}(v, n)=0$.
In fact, we can calculate
$\displaystyle
\hat{\pi}^* \mathcal{F}_n \chi_r
= \hat{\pi}^* \left(\sum_{s=0}^n \Psi_{n}^{(1)}(s,r)\chi_s + \sum_{s=1}^n \Psi_n^{(3)}(s,r) \,^{\sgn}\chi_s \right)
= \sum_{v=0}^{n-1} \Psi_{n}^{(1)}(v+1,r)\chi_{v} + \sum_{v=1}^n \Psi_n^{(3)}(v+1,r) \,^{\sgn}\chi_{v}\;
$
(Remark we used $\Psi_{n}^{(3)}(1,r)=0$ by Prop.\ref{psi-zero}(2)),\;
and 
$\displaystyle
\mathcal{F}_{n-1} \pi_* \chi_r
= \mathcal{F}_{n-1} \left( -\gamma^{r-1} \,^{\sgn}\chi_{r-1} +
\gamma^{r} \,^{\sgn}\chi_{r}  \right) \\
= \sum_{v=0}^{n-1} \left( -\gamma^{r-1} \Psi^{(2)}_{n-1}(v,r-1) +\gamma^r \Psi^{(2)}_{n-1}(v,r) \right) \chi_v +
\sum_{v=1}^{n-1} \left( -\gamma^{r-1} \Psi^{(4)}_{n-1}(v,r-1) +\gamma^r \Psi^{(4)}_{n-1}(v,r) \right)\,^{\sgn}\chi_v.
$
Then comparing coefficients of bases, we get (\ref{jun-1}) and (\ref{jun-2}).
Likewise for $1 \leq r \leq n$, an equation $\hat{\pi}^* \mathcal{F}_n (^{\sgn}\chi_r)= \mathcal{F}_{n-1} \pi_* (^{\sgn}\chi_r)$ induces
\begin{eqnarray}
 \label{jun-3-0}
&&\Psi^{(2)}_n(1, r) + \Psi^{(4)}_n(1, r)
=\; \gamma^r \left( \Psi^{(1)}_{n-1}(0,r) + \Psi^{(1)}_{n-1}(0, r-1)\right),\\
 \label{jun-3}
&&\Psi^{(2)}_n(v+1, r)=\; \gamma^r \left( \Psi^{(1)}_{n-1}(v,r) + \Psi^{(1)}_{n-1}(v, r-1)\right)\qquad(1 \leq v \leq n-1), \quad\text{and}\\
 \label{jun-4}
&&\Psi^{(4)}_n(v+1, r)=\; \gamma^r \left( \Psi^{(3)}_{n-1}(v,r) +  \Psi^{(3)}_{n-1}(v, r-1) \right)\qquad(1 \leq v \leq n-1),
\end{eqnarray}
where $\Psi^{(1)}_{n-1}(v,n)=\Psi^{(3)}_{n-1}(v,n)=0$.
Next, remember the matrix of inverse Fourier transformation $\bar{\mathcal{F}_n}$ is $\frac{1}{|\S_n|} \overline{\Psi_n}$(complex conjugate) by (\ref{bar-change-me}).
Then for $1 \leq s \leq n$, an equation $\pi_* \bar{\mathcal{F}_n} (\chi_s)= \bar{\mathcal{F}}_{n-1} \hat{\pi}^* (\chi_s)$ induces
\begin{eqnarray}
 \label{gyaku-1}
&& \overline{\gamma}^{u+1} \Psi^{(3)}_{n}(u+1,s) +  \overline{\gamma}^{u} \Psi^{(3)}_{n}(u, s)
=\;  q^n \Psi^{(1)}_{n-1}(u, s-1)
\qquad(1 \leq u \leq n-1), \quad\text{and}\\
 \label{gyaku-2}
&& \overline{\gamma}^{u} \left( \Psi^{(1)}_{n}(u+1,s) -  \Psi^{(1)}_{n}(u, s) \right)
=\; - q^n \Psi^{(3)}_{n-1}(u, s-1)
\qquad(0 \leq u \leq n-1),
\end{eqnarray}
where $\Psi^{(3)}_{n-1}(0,s-1)=\Psi^{(1)}_{n-1}(0,s-1)$.
Finally for $1 \leq s \leq n$, an equation $\pi_* \bar{\mathcal{F}_n} (^{\sgn}\chi_s)= \bar{\mathcal{F}}_{n-1} \hat{\pi}^* (^{\sgn}\chi_s)$ induces
\begin{eqnarray}
 \label{gyaku-3-0}
&& \Psi^{(2)}_{n}(1,s) - \Psi^{(2)}_{n}(0,s) -  \overline{\gamma} \Psi^{(4)}_{n}(1, s)
=\;  - q^n \Psi^{(2)}_{n-1}(0, s-1),
\\
 \label{gyaku-3}
&& \overline{\gamma}^{u+1} \Psi^{(4)}_{n}(u+1,s) +  \overline{\gamma}^{u} \Psi^{(4)}_{n}(u, s)
=\;  q^n \Psi^{(2)}_{n-1}(u, s-1)
\qquad(1 \leq u \leq n-1), \quad\text{and}\\
 \label{gyaku-4}
&& \overline{\gamma}^{u} \left( \Psi^{(2)}_{n}(u+1,s) -  \Psi^{(2)}_{n}(u, s) \right)
=\; - q^n \Psi^{(4)}_{n-1}(u, s-1)
\qquad(1 \leq u \leq n-1),
\end{eqnarray}
where consider $\Psi^{(2)}_{n-1}(u, 0)=1,\;\Psi^{(4)}_{n-1}(u, 0)=0$.

Using above induced relations from (\ref{jun-1}) to (\ref{gyaku-4}), we can prove some properties of $\Psi=\Psi_n$ as follows:

\begin{pp}
 \label{psi-zero2}
\; For even $s$ and even $r$,\;\;$\displaystyle\Psi^{(4)}(s,r)=0$
(Combining Prop.\ref{psi-zero}(3), we know $\Psi^{(4)}(s,r)=0$ unless both $s$ and $r$ are odd).
\end{pp}
\begin{proof}
Use (\ref{jun-4}) and Prop.\ref{psi-zero}(2), or (\ref{gyaku-4}) and Prop.\ref{psi-zero}(1).
\end{proof}

\begin{pp}
 \label{rinsetu}
\;{\rm(1)}\; For odd $s \;( 1 \leq s \leq n-1 )$ and any $r$,\;\;
$\Psi^{(1)}(s,r)=\Psi^{(1)}(s+1,r)$.\\
\quad \;{\rm(2)}\; For $s\geq 1$ and even $r \;(0 \leq r \leq n)$,\;\;$\Psi^{(1)}(s,r)= -\Psi^{(1)}(s,r+1)$\;(where $\Psi^{(1)}_n(s,n+1)=0$).\\
\quad \;{\rm(3)}\; For even $s \; (0 \leq s \leq n-1)$ and any $r$,\;\;$\Psi^{(2)}(s,r)= \Psi^{(2)}(s+1,r)$.\\
\quad \;{\rm(4)}\; For any $s$ and odd $r \;(1 \leq r \leq n)$,\;\;$\Psi^{(3)}(s,r)= -\Psi^{(3)}(s,r+1)$\;(where $\Psi^{(3)}_n(s,n+1)=0$).
\end{pp}
\begin{proof}
\;(1)\; Use (\ref{gyaku-2}) and Prop.\ref{psi-zero}(2) (except for the case $r=0$ by (\ref{psi-s-0}) ).
\;(2)\; By (\ref{jun-3}) and Prop.\ref{psi-zero}(1).
\;(3)\;  When $s \neq 0$, (\ref{gyaku-4}) and Prop.\ref{psi-zero2} imply.
When $r$ is odd, both sides are $0$ by Prop.\ref{psi-zero}(1). 
And when $s=0$ and $r$ is even, (\ref{gyaku-3-0}), Prop.\ref{psi-zero}(1) and (3) imply.
\;(4)\; By (\ref{jun-4}) and Prop.\ref{psi-zero2}.
\end{proof}



\subsection{Commutative diagram for the pairs $(\S_n, \mathbb{F}^\times \times \G_n)$ and $(\S_{n-2}, \mathbb{F}^\times \times \G_{n-2})$}

In this subsection, we abbreviate $\mathbb{F}^\times \times \G_n$ to $G_n$, and investigate a relation between the pairs $(\S_n, G_n)$ and $(\S_{n-2}, G_{n-2})$, where $n\geq 2$.

Let $\pi: \S_n \rightarrow \S_{n-2},\; (b\,||\, z, \,^{t}\!z, w) \mapsto b$, where the notation $(b \,||\, z, \,^{t}\!z, w)\;(b \in \S_{n-2},\; z \in \M_{n-2, 2},\; w \in \S_2)$ is the same as in (\ref{notation2}).
Take an intersection character corresponding to $\displaystyle e=(0\,||\,\,0,0,\; \left(\begin{array}{cc}\\[-10pt]
0& 1 \\[0pt]
1& 0  \end{array}\right)) \in \S_{n}$ and define $\pi_*=\pi_*^e$ and $\hat{\pi}=\hat{\pi}_e$.
Then since $\displaystyle {}^{t}\!\pi(b)+e=(b\,||\,0,0,\; \left(\begin{array}{cc}\\[-10pt]
0& 1 \\[0pt]
1& 0  \end{array}\right))\;(b \in \S_{n-2})$, the rank of it is bigger by 2 than that of $b$.
And since $\sgn\left(\begin{array}{cc}\\[-10pt]
0& 1 \\[0pt]
1& 0  \end{array}\right)=\epsilon$, we have $\sgn({}^{t}\!\pi(b)+e)=\epsilon \cdot \sgn(b)$ by (\ref{sgn-seki}).
Therefore we can regard
$\hat{\pi}: \Lambda_{n-2}' \rightarrow \Lambda_n'$
as a map between parameters in (\ref{lambda-dash}) by
\begin{equation}
\label{s-hatpi2}
\hat{\pi}(0)=2^\epsilon,\quad
\hat{\pi}(v)=v+2 \;(\text{$v$ is odd}),\quad
\hat{\pi}(v^{\pm})=(v+2)^{\pm\epsilon}\; (\text{$v$ is even, double-sign corresponds}).
\end{equation}
Especially ($\natural$) holds.
So we get a commutative diagram (\ref{d2}) as follows:
\begin{equation}
 \label{s-d2}
\begin{CD}
\mathbb{C}[_{G_{n}}\!\backslash \S_{n}] @>{\mathcal{F}_{n}}>>\mathbb{C}[_{G_{n}}\!\backslash (\S_{n})\,\hat{}\;] \\
@V{\pi_{*}}V{\qquad\qquad\quad\;\;\circlearrowright}V  @VV{\hat{\pi}^*}V \\
\mathbb{C}[_{\G_{n-2}}\!\backslash \S_{n-2}] @>>{\mathcal{F}_{n-2}}> \mathbb{C}[_{\G_{n-2}}\!\backslash (\S_{n-2})\,\hat{}\;]
\end{CD}
\end{equation}

Let $\displaystyle \Delta=\left(\begin{array}{cc}\\[-12pt]
\Delta^{(1)}(v,s)& \Delta^{(2)}(v,s) \\[0pt]
\Delta^{(3)}(v,s)& \Delta^{(4)}(v,s)  \end{array}\right)$
be the matrix of $\hat{\pi}^*$ on the bases $\hat{\mathcal{B}}'_n$ and $\hat{\mathcal{B}}'_{n-2}$,
that is,
$\displaystyle
\hat{\pi}^* (\chi_s)= \sum_{v=0}^{n-2} \Delta^{(1)}(v,s)\chi_v+\sum_{v\geq2, \text{even}} \Delta^{(3)}(v,s)\,^{\sgn}\chi_v\;\;(0 \leq s \leq n),\;
$
and\;
$\displaystyle
\hat{\pi}^*\,(^{\sgn}\chi_s)= \sum_{v=0}^{n-2} \Delta^{(2)}(v,s)\chi_v+\sum_{v\geq2, \text{even}} \Delta^{(4)}(v,s)\,^{\sgn}\chi_v\;\;(2 \leq s \leq n,\; \text{even})
$.
Then we have
\begin{eqnarray}
 \label{s-Delta2}
\begin{split}
\Delta^{(1)}(v,s)=\;\delta_{v+2,s},\quad
&\Delta^{(2)}(v,s)=\;\epsilon \cdot \delta_{v,0} \cdot\delta_{s,2},\\
\Delta^{(3)}(v,s)=\;0,\quad
&\Delta^{(4)}(v,s)=\;\epsilon \cdot \delta_{v+2,s},
\end{split}
\end{eqnarray}
by easy calculations using (\ref{s-hatpi2}) (Refer to calculations of (\ref{s-Delta1}) also).

On the other hand, 
let $\displaystyle E=\left(\begin{array}{cc}\\[-12pt]
E^{(1)}(u,r)& E^{(2)}(u,r) \\[0pt]
E^{(3)}(u,r)& E^{(4)}(u,r)  \end{array}\right)$
be the matrix of $\pi_*$ on the bases $\mathcal{B}'_n$ and $\mathcal{B}'_{n-2}$,
that is,
$\displaystyle
{\pi}_* (\chi_r)= \sum_{u=0}^{n-2} E^{(1)}(u,r)\chi_u+\sum_{u\geq 2, \text{even}} E^{(3)}(u,r)\,^{\sgn}\chi_u\;\;(0 \leq r \leq n),\;
$
and
$\displaystyle
{\pi}_*\,(^{\sgn}\chi_r)= \sum_{u=0}^{n-2} E^{(2)}(u,r)\chi_u+\sum_{u\geq 2, \even} E^{(4)}(u,r)\,^{\sgn}\chi_u\;\;(2 \leq r \leq n,\; \even)
$.
Then we have

\begin{eqnarray}
 \label{s-E2}
\begin{split}
&E^{(1)}(u,u)=q^u,\quad E^{(1)}(u,u+1)= q^u(q-1),\quad E^{(1)}(u,u+2)= -q^{u+1} \quad(0 \leq u \leq n-2),\\
&E^{(2)}(0,2)= -\epsilon q,\\
&E^{(4)}(u,u)=q^u,\quad E^{(4)}(u,u+2)= -\epsilon q^{u+1} \quad(2 \leq u \leq n-2,\;\even),\\
\text{and}\;&\; \text{except for the above,}\;\; E^{(i)}(u,r)=0\quad(i=1,2,3,4).
\end{split}
\end{eqnarray}
We calculate it below (similarly to (\ref{s-E1})).
Let's use $b^+= \diag(1, \cdots, 1,0,\cdots,0) \in \mathcal{O}_{n-2}(u^+)$ and $b^-= \diag(1, \cdots, 1,\delta^{-1},0,\cdots,0) \in \mathcal{O}_{n-2}(u^-)$ as representatives of orbits.
When writing $z=\binom{z'}{z''},\; z' \in \M_{u,2},\; z'' \in \M_{n-u-2,2}$ for $z \in \M_{n-2,2}$,
if $z'' \neq 0$, we have
$\displaystyle
\;^{\sgn}\chi_r \left(  b^{\pm} \,||\, z, \,^t\!z,\; \; \left(\begin{array}{cc}\\[-10pt]
x& w \\[0pt]
w& y  \end{array}\right)  \right)
=\;
^{\sgn}\chi_r \left(  b^{\pm} \,||\, z, \,^t\!z,\; \; \left(\begin{array}{cc}\\[-10pt]
x& 0 \\[0pt]
0& y  \end{array}\right)  \right)
$
by row and column operations and an adequate replacement of variable.
Therefore,
\begin{eqnarray}
 \begin{split}\nonumber
\pi_*\chi_r(u^{(+)})&=\; \sum_{z \in \M_{n-2,2}}\; \sum_{x,y,w \in \mathbb{F}} \overline{\theta}(2w) \chi_r \left(  b^{+}\,||\, z, \,^t\!z,\; \left(\!\begin{array}{cc}\\[-15pt]
x& w \\[0pt]
w& y  \end{array}\right)  \right)
\\
&=\;\sum_{z' \in \M_{u,2}}\; \sum_{x,y,w \in \mathbb{F}} \overline{\theta}(2w) \chi_r \left(  b^{+}\,||\, z', \,^t\!z',\; \left(\!\begin{array}{cc}\\[-15pt]
x& w \\[0pt]
w& y  \end{array}\right)  \right)
\\
&=\; \;\sum_{z' \in \M_{u,2}}\; \sum_{x,y,w \in \mathbb{F}} \overline{\theta}(2w) \chi_{r-u}\left(\!\begin{array}{cc}\\[-15pt]
x-\sum z_{i1}^2& w-\sum z_{i1}z_{i2} \\[0pt]
w-\sum z_{i1}z_{i2}& y-\sum z_{i2}^2  \end{array}\right)
\\
&=\; \;\sum_{z' \in \M_{u,2}}\; \sum_{x,y,w \in \mathbb{F}} \overline{\theta}\left(2(w+\sum z_{i1}z_{i2})\right) \chi_{r-u}\left(\!\begin{array}{cc}\\[-15pt]
x& w \\[0pt]
w& y  \end{array}\right)
\quad
=\; q^u\,\sum_{x,y,w \in \mathbb{F}}\overline{\theta}(2w) \chi_{r-u}\left(\!\begin{array}{cc}\\[-15pt]
x& w \\[0pt]
w& y  \end{array}\right),
 \end{split}
\end{eqnarray}
where $u^{(+)}$ denote $u\;(u=0$ or $\odd$) or $u^+\; (u\geq 2, \even)$, and $z'=(z_{ij} \mid 1\leq i \leq u,\; j=1,2) \in \M_{u,2}$.
Likewise we have
$\displaystyle
\pi_*\chi_r(u^-)
=\; q^u\,\sum_{x,y,w \in \mathbb{F}}\overline{\theta}(2w) \chi_{r-u}\left(\!\begin{array}{cc}\\[-15pt]
x& w \\[0pt]
w& y  \end{array}\right)
$
\;and\;
$\displaystyle
\pi_*\;^{\sgn}\chi_r(u^{(\pm)})
=\; \pm q^u\,\sum_{x,y,w \in \mathbb{F}}\overline{\theta}(2w)\, ^{\sgn}\chi_{r-u}\left(\!\begin{array}{cc}\\[-15pt]
x& w \\[0pt]
w& y  \end{array}\right)\;
$
(double-sign corresponds).
These equations imply that $E^{(i)}(u,r)=0\; (i=1,2,3,4)$ unless $r = u,u+1,u+2$, and induce all values of all $E^{(i)}(u,r)$ easily.
For example,
$\displaystyle
E^{(1)}(u,u+2)
= \pi_*\chi_{u+2}(u^{(+)})
= \; q^u\sum_{x,y,w \in \mathbb{F}}\overline{\theta}(2w) \chi_{2}\left(\!\begin{array}{cc}\\[-15pt]
x& w \\[0pt]
w& y  \end{array}\right)
=\; -q^{u+1},\;
$
or
$\displaystyle
E^{(3)}(u,r)=\;\frac{1}{2} \left( \pi_* \chi_r(u^+) - \pi_* \chi_r(u^-) \right)=0\;
$
at once,
and so on.

Now using the diagram(\ref{s-d2}), the matrices $\Delta$ (\ref{s-Delta2}) and $E$ (\ref{s-E2}), we get some relations between $\Psi_n$ and $\Psi_{n-2}$.
Remark we skip relations about $\Psi^{(4)}$, since we already know $\Psi^{(4)}(s,r)=0$ for $\even s,r$ by Prop.\ref{psi-zero2}.
Consider $\Psi^{(i)}_n(s,r)=0$ whenever $r \leq -1$ or $r \geq n+1$ in the following equations.
For $0 \leq r \leq n$, an equation
$\hat{\pi}^* \mathcal{F}_n (\chi_r)= \mathcal{F}_{n-2} \pi_* (\chi_r)$
induces
\begin{eqnarray}
 \label{2jun-1-0}
&&
\Psi^{(1)}_n(2, r)+ \epsilon \Psi^{(3)}_n(2, r)=\; -q^{r-1} \Psi^{(1)}_{n-2}(0,r-2) + q^{r-1}(q-1) \Psi^{(1)}_{n-2}(0,r-1) +q^{r} \Psi^{(1)}_{n-2}(0,r),
\\
 \label{2jun-1}
&&
\Psi^{(1)}_n(v+2, r)=\; -q^{r-1} \Psi^{(1)}_{n-2}(v,r-2) + q^{r-1}(q-1) \Psi^{(1)}_{n-2}(v,r-1) +q^{r} \Psi^{(1)}_{n-2}(v,r)\qquad(1 \leq v \leq n-2),
\\
 \label{2jun-2}
&&
\begin{split}
\epsilon \Psi^{(3)}_n(v+2, r)=\; -q^{r-1} \Psi^{(3)}_{n-2}(v,r-2) + q^{r-1}(q-1) \Psi^{(3)}_{n-2}(v,r-1) +q^{r} \Psi^{(3)}_{n-2}(v,r)\qquad(2 \leq v \leq n-2, \, \even).\\{}\\{}
\end{split}
\end{eqnarray}

\vspace{-0.3in}
\noindent
For $2 \leq r \leq n\;(\even)$, an equation $\hat{\pi}^* \mathcal{F}_n (^{\sgn}\chi_r)= \mathcal{F}_{n-2} \pi_* (^{\sgn}\chi_r)$ induces
\begin{equation}
 \label{2jun-3}
\Psi^{(2)}_n(v+2, r)=\; -\epsilon q^{r-1} \Psi^{(2)}_{n-2}(v,r-2) + q^{r} \Psi^{(2)}_{n-2}(v,r)\qquad(0 \leq v \leq n-2),
\end{equation}
where consider $\Psi^{(2)}_{n-2}(v,0)=1$.
Next using inverse Fourier transformations, for $0 \leq s \leq n$, an equation $\pi_* \bar{\mathcal{F}_n}(\chi_s)= \bar{\mathcal{F}}_{n-2} \hat{\pi}^* (\chi_s)$ induces
\begin{eqnarray}
 \label{2gyaku-1-0}
&&
\Psi^{(1)}_{n}(0,s) - \Psi^{(1)}_{n}(2,s) -\epsilon q \Psi^{(3)}_{n}(2,s)
=\;  q^{2n -1} \Psi^{(1)}_{n-2}(0, s-2),
\\
 \label{2gyaku-1}
&&
q^u \Psi^{(1)}_{n}(u,s) + q^u(q-1)\Psi^{(1)}_{n}(u+1,s) -q^{u+1} \Psi^{(1)}_{n}(u+2,s)
\\ \nonumber
&&\qquad\qquad\qquad\qquad\qquad\qquad
=\;  q^{2n -1} \Psi^{(1)}_{n-2}(u, s-2)
\quad(1 \leq u \leq n-2),
\\
 \label{2gyaku-2}
&&
q^u \Psi^{(3)}_{n}(u,s)  -\epsilon q^{u+1} \Psi^{(3)}_{n}(u+2,s)
=\;  q^{2n -1} \Psi^{(3)}_{n-2}(u, s-2)
\qquad(2 \leq u \leq n-2,\; \even).
\end{eqnarray}
Remark we used Prop.\ref{rinsetu}(1) together for (\ref{2gyaku-1-0}).
Finally for $2 \leq s \leq n\;(\even)$, an equation $\pi_* \bar{\mathcal{F}_n} (^{\sgn}\chi_s)= \bar{\mathcal{F}}_{n-2} \hat{\pi}^* (^{\sgn}\chi_s)$ induces
\begin{equation}
 \label{2gyaku-3}
q^u \Psi^{(2)}_{n}(u,s) + q^u(q-1)\Psi^{(2)}_{n}(u+1,s) -q^{u+1} \Psi^{(2)}_{n}(u+2,s)
=\;  \epsilon q^{2n -1} \Psi^{(2)}_{n-2}(u, s-2)
\qquad(0 \leq u \leq n-2),
\end{equation}
where consider $\Psi^{(2)}_{n-1}(u, 0)=1$.

\begin{pp}
 \label{gamma-jijo}
\; $\gamma^2=\epsilon q$.
\end{pp}
\begin{proof}
\; Put $n=2$ and $r=1$ in (\ref{2jun-1}), we have $\3_2(2,1)= \epsilon q$.
Substituting it in (\ref{gyaku-1}), we get the result. 
\end{proof}


\subsection{Values of matrix $\Psi$}
\label{5-4}

At the end of our investigation, using results in previous subsections together, we try to represent all values of  matrix elements of  $\Psi=\Psi_n$ (\ref{Psi}) in $q$-Krawtchouk polynomials.
Here $\lfloor \cdot \rfloor$ denotes the floor function.

Firstly as for $\Psi^{(1)}(s,r) \;(s \neq 0)$, results are as follows:
\begin{pp}
 \label{psi1}
\; For $n \geq 1,\;1 \leq s \leq n$ and $0\leq r \leq n$,
$$
\1_{n}(s,r)=\;(-1)^{r+x} q^{x(x+1)}
(q^{2N+(-1)^n}; q^{-2})_x \begin{bmatrix} N \\ x \end{bmatrix}_{x^2}
\, K^{\mathrm{Aff}}_y(x; q^{-2N-(-1)^n}, N; q^2),
$$\vspace{-0.1in}
$$\text{where}\quad N=\lfloor \frac{n-1}{2} \rfloor,\;y=\lfloor \frac{s-1}{2} \rfloor,\; x=\lfloor \frac{r}{2} \rfloor.
$$
\end{pp}
\begin{proof}\;
First consider $\1(s,r)$ for $\odd s$ and $\even r$.
Then (\ref{2jun-1}) is changed by Prop.\ref{rinsetu}(2) as follows:
\begin{equation} \tag{\ref{2jun-1}'}
\Psi^{(1)}_n(v+2, r)=\; -q^{r} \Psi^{(1)}_{n-2}(v,r-2) + q^{r} \Psi^{(1)}_{n-2}(v,r)\qquad(1 \leq v \leq n-2,\odd\; \text{and}\; 0\leq r \leq n,\even).
\end{equation}
Likewise, (\ref{2gyaku-1}) is changed by Prop.\ref{rinsetu}(1) as follows:
\begin{equation} \tag{\ref{2gyaku-1}'}
q^{u+1} \Psi^{(1)}_{n}(u,s) - q^{u+1} \Psi^{(1)}_{n}(u+2,s)
=\;  q^{2n -1} \Psi^{(1)}_{n-2}(u, s-2)
\quad(1 \leq u \leq n-2,\odd\; \text{and}\; 0\leq s \leq n,\even).
\end{equation}
Define two families $\{f_{N}\}_{N \geq 0}$ and $\{g_{N}\}_{N \geq 0}$ by
$$
f_{N}(y,x)= \1_{2N+1}(2y+1,2x), \quad g_N(y,x)= \1_{2N+2}(2y+1,2x)\qquad (N \geq 0,\; 0 \leq x,\!y \leq N).
$$
Then (\ref{2jun-1}') and (\ref{2gyaku-1}') are regarded as
BPR(\ref{koutai}) and FPR(\ref{zensin}) for $\{f_N\}$ of the case $a=1,\;b=q^2,\;c=1,\; d=q,\; t=q^2$ and $\sigma=\1_{1}(1,0)=1$, and for $\{g_N\}$ of the case $a=1,\;b=q^2,\;c=1,\; d=q^3,\; t=q^2$ and $\sigma=\1_{2}(1,0)=1$.
So we conclude by Prop.\ref{zenkasikitoku}(3),
\begin{equation}
\label{psi1-1-r}
\1_{n}(1,2x)=\; \;(-1)^x q^{x(x+1)} \frac{(q^{n-1};q^{-1})_{2x}}{(q^2;q^2)_x}\qquad(0 \leq 2x \leq n),\\
\end{equation}
\vspace{-0.25in}
\begin{eqnarray}
\begin{split}
&\1_{2N+1}(2y+1,2x)
=\; \1_{2N+1}(1,2x) \; K^{\mathrm{Aff}}_y(x; q^{-2N+1}, N; q^2), \qquad(0 \leq y,\!x \leq N)\;\; \text{and}\\
&\1_{2N+2}(2y+1,2x)
=\; \1_{2N+2}(1,2x) \; K^{\Aff}_y(x; q^{-2N-1}, N; q^2)\qquad(0 \leq y,x \leq N).
\end{split}
\end{eqnarray}
And by Prop.\ref{rinsetu}(1) and (2), all results in Prop.\ref{psi1} are gotten.
\end{proof}

Next let's solve $\Psi^{(1)}(0,r)\;(0 \leq r \leq n)$ for completing $\1$.
\begin{pp}
 \label{1shoki}
\;$\displaystyle (1)\qquad 
\1_n(0,2x)
=\; (-1)^x q^{x(x+1)} \frac{(q^n;q^{-1})_{2x}}{(q^2;q^2)_x}
\qquad (0 \leq 2x \leq n),
$
\\
$$\displaystyle (2)\qquad
\1_n(0,2x+1)
=\; (-1)^{x+1} q^{x(x+1)} \frac{(q^n;q^{-1})_{2x+1}}{(q^2;q^2)_x}\qquad (1 \leq 2x+1 \leq n).
$$
\end{pp}
\begin{lem}
 \label{rinsetu-2}
\;For $\even r\;(0 \leq r \leq n)$, $\displaystyle \Psi^{(1)}_n(0,r) + \Psi^{(1)}_n(0,r+1)= \epsilon q^{-r-1} \,\3_{n+2}(2,r+1)$.
\end{lem}
\begin{proof}\;
We have
$\displaystyle \4_{n+1}(1,r+1)= \gamma^{r+1}\left( \1_n(0,r) + \1_n(0,r+1) \right)$ by (\ref{jun-3-0}).
And $\displaystyle \3_{n+2}(2,r+1)=\gamma^{r+1} \4_{n+1}(1,r+1)$
by (\ref{jun-2}).
Using these equations with $\gamma^2=\epsilon q$ (Lem.\ref{gamma-jijo}), we get the result.
\end{proof}
\noindent
\underline{\it Proof} of Prop.\ref{1shoki}.\;\;
When $r$ is even, (\ref{2jun-1-0}) is changed as\; \\
$\displaystyle
\1_n(2,r)+ \epsilon \3_n(2,r)
= q^{r-1}(q-1) \left( \1_{n-2}(0,r-2) + \1_{n-2}(0,r-1) \right) -q^r \1_{n-2}(0,r-2) + q^r \1_{n-2}(0,r).$
And by the lemme, we have
$\displaystyle
\1_n(2,r)+ \epsilon q \3_n(2,r)
=  -q^r \1_{n-2}(0,r-2) + q^r \1_{n-2}(0,r).$
Substituting it in (\ref{2gyaku-1}), we have
\begin{equation}
 \label{1shoki-even}
\1_n(0,r)= q^r \1_{n-2}(0,r) + (q^{2n-1} - q^r)\,\1_{n-2}(0,r-2)\quad (0 \leq r \leq n,\; \even).
\end{equation}
We can apply the recursion (\ref{o-zenka}) to this (according to the parity of $n$), and we get Prop.\ref{1shoki}(1).
On the other hand, we have 
$\displaystyle
\1_n(0,r)= \1_n(1,r) + q^n \1_{n-1}(0,r-1)\;(r \geq 1)
$
by (\ref{gyaku-2}).
Substituting $r=2x+1$ and using Prop.\ref{psi1} and \ref{1shoki}(1), we have (2).
\vspace{-0.2in}
\begin{flushright}
$\Box$
\end{flushright}

$\2$ is also solved as follows (with Prop.\ref{psi-zero}(1), we complete $\2$):
\begin{pp}
 \label{psi2}
\; For $n \geq 2,\;0 \leq s \leq n$ and $\even r,\;2\leq r \leq n$,
$$
\2_{n}(s,r)=\;(-\epsilon)^{x} q^{x^2}
(q^{2N-(-1)^n}; q^{-2})_x \begin{bmatrix} N \\ x \end{bmatrix}_{x^2}
\, K^{\mathrm{Aff}}_y(x; q^{-2N+(-1)^n}, N; q^2),
$$\vspace{-0.1in}
$$\text{where}\quad N=\lfloor \frac{n}{2} \rfloor,\;y=\lfloor \frac{s}{2} \rfloor,\; x= \frac{r}{2}.
$$
\end{pp}
\begin{proof}\;
First we consider $\2(s,r)$ for $s$ is also even.
Then (\ref{2gyaku-3}) is changed by Prop.\ref{rinsetu}(3) as follows:
\begin{equation} \tag{\ref{2gyaku-3}'}
q^{u+1} \Psi^{(2)}_{n}(u,s) -q^{u+1} \Psi^{(2)}_{n}(u+2,s)
=\;  \epsilon q^{2n -1} \Psi^{(2)}_{n-2}(u, s-2)
\qquad(0 \leq u \leq n-2,\even\; \text{and}\; 2\leq s \leq n,\even).
\end{equation}
(\ref{2jun-3}) and (\ref{2gyaku-3}') give 
BPR(\ref{koutai}) and FPR(\ref{zensin}) for $f_N(y,x)=\2_{2N}(2y,2x)$ of the case $a=1,\;b=\epsilon q,\;c=1,\; d=\epsilon,\; t=q^2,\;\sigma=1$,
and for $g_N(y,x)=\2_{2N+1}(2y,2x)$ of the case $a=1,\;b=\epsilon q,\;c=1,\; d=\epsilon q^2,\; t=q^2,\;\sigma=1$.
Therefore by Prop.\ref{zenkasikitoku}(3),
\begin{equation}
\label{psi2-0-r}
\2_{n}(0,2x)=\; \;(-\epsilon)^x q^{x^2} \frac{(q^n; q^{-1})_{2x}}{(q^2;q^2)_x}\qquad(0 \leq 2x \leq n),\\
\end{equation}
\vspace{-0.25in}
\begin{eqnarray}
\begin{split}
&\2_{2N}(2y,2x)
=\; \2_{2N}(0,2x) \; K^{\mathrm{Aff}}_y(x; q^{-2N+1}, N; q^2), \qquad(0 \leq y,\!x \leq N)\;\; \text{and}\\
&\2_{2N+1}(2y,2x)
=\; \2_{2N+1}(0,2x) \; K^{\Aff}_y(x; q^{-2N-1}, N; q^2)\qquad(0 \leq y,\!x \leq N),
\end{split}
\end{eqnarray}
where consider $\2(s,0)=1$.
And by Prop.\ref{rinsetu}(3), all results in Prop.\ref{psi2} are gotten.
\end{proof}

Next, $\3$ is as follows (with (\ref{psi-s-0}) and Prop.\ref{psi-zero}(2), we complete $\3$):
\begin{pp}
 \label{psi3}
\; For $n \geq 2,\;\even s,\; 2 \leq s \leq n$ and $1\leq r \leq n$,
$$
\3_{n}(s,r)= 
\;(-1)^{r+x+1} \epsilon^{y+1} q^{n+x^2+x-y-1}
(q^{2N-(-1)^n}; q^{-2})_x \begin{bmatrix} N \\ x \end{bmatrix}_{x^2}
\, K^{\mathrm{Aff}}_y(x; q^{-2N+(-1)^n}, N; q^2),
$$\vspace{-0.1in}
$$\text{where}\quad N=\lfloor \frac{n-2}{2} \rfloor,\;y=\frac{s-2}{2},\; x=\lfloor \frac{r-1}{2} \rfloor.
$$
\end{pp}

\begin{proof}\;
First we consider $\3(s,r)$ for $\even s$ and $\odd r$.
Then by Prop.\ref{rinsetu}(4) we have
\begin{equation} \tag{\ref{2jun-2}'}
\epsilon \Psi^{(3)}_n(v+2, r)=\; -q^{r} \Psi^{(3)}_{n-2}(v,r-2) + q^{r} \Psi^{(3)}_{n-2}(v,r)\qquad(2 \leq v \leq n-2,\even\; \text{and}\; 1\leq r \leq n,\odd).
\end{equation}
(\ref{2jun-2}') and (\ref{2gyaku-2}) give 
BPR(\ref{koutai}) and FPR(\ref{zensin}) for
$f_N(y,x)=\3_{2N+2}(2y+2,2x+1)$ of the case $a=\epsilon q,\;b=\epsilon q^3,\;c=\epsilon q^{-1},\; d=\epsilon q^2,\; t=q^2$ and $\sigma=\3_2(2,1)=\epsilon q$ (from the proof of Lem.\ref{gamma-jijo}),
and for $g_N(y,x)=\3_{2N+3}(2y+2,2x+1)$ of the case $a=\epsilon q,\;b=\epsilon q^3,\;c=\epsilon q^{-1},\; d=\epsilon q^4,\; t=q^2$ and $\sigma=\3_3(2,1)=\epsilon q^2$ (by Lem.\ref{rinsetu-2}).
Therefore by Prop.\ref{zenkasikitoku}(3),
\begin{equation}
\label{psi3-2-r}
\3_{n}(2,2x+1)=\; \;(-1)^x \epsilon q^{n+x^2+x-1} \frac{(q^{n-2}; q^{-1})_{2x}}{(q^2;q^2)_x}\qquad(1 \leq 2x+1 \leq n-1),\\
\end{equation}
\vspace{-0.25in}
\begin{eqnarray}
\begin{split}
&\3_{2N+2}(2y+2,2x+1)
=\; \epsilon^y q^{-y}\, \3_{2N+2}(2,2x+1) \; K^{\mathrm{Aff}}_y(x; q^{-2N+1}, N; q^2), \qquad(0 \leq y,\!x \leq N)\;\; \text{and}\\
&\3_{2N+3}(2y+2,2x+1)
=\; \epsilon^y q^{-y}\, \3_{2N+3}(2,2x+1) \; K^{\Aff}_y(x; q^{-2N-1}, N; q^2)\qquad(0 \leq y,\!x \leq N).
\end{split}
\end{eqnarray}
And by Prop.\ref{rinsetu}(4), all results in Prop.\ref{psi3} are gotten.
\end{proof}

Finally we compute $\4$ as follows (with Prop.\ref{psi-zero2} we complete $\4$):
\begin{pp}
 \label{psi4}
\; For $n \geq 1,\;\odd s$ and $r,\; 1 \leq s,r \leq n$,
$$
\4_{n}(s,r)= 
\;(-1)^{x} \epsilon^{x+y} q^{n+x^2-y-1} \gamma
(q^{2N+(-1)^n}; q^{-2})_x \begin{bmatrix} N \\ x \end{bmatrix}_{x^2}
\, K^{\mathrm{Aff}}_y(x; q^{-2N-(-1)^n}, N; q^2),
$$\vspace{-0.1in}
$$\text{where}\quad N=\lfloor \frac{n-1}{2} \rfloor,\;y=\frac{s-1}{2},\; x= \frac{r-1}{2}.
$$
\end{pp}
\begin{proof}
By (\ref{jun-2}), we have $\4_n(s,r)= \gamma^{-r} \3_{n+1}(s+1, r)$.
Substituting Prop.\ref{psi3}, we get the result at once.
\end{proof}
\noindent
So we had all values of matrix elements $\Psi$.
And we can calculate canonical matrix elements $\varPhi$ also by (\ref{phi-psi}).


\section*{Concluding remarks}
In this paper, we found Krawtchouk or Affine $q$-Krawtchouk polynomials as matrix elements of group-invariant Fourier transformations, or zonal spherical functions (by relations in Sec.\ref{1-3}) on $\mathbb{F}^n$, $\M_n$, $\A_n$ and $\S_n$ on a finite field.
These examples are all in finite settings.
But we can extend some of our investigations to the pairs of locally compact Abelian groups and acting compact groups using the Haar measure instead of counting.
Then we will see some kinds of ($q$-)Krawtchouk polynomials as zonal spherical functions of some examples. 
It may be seen in the following papers.

\section*{Acknowledgement}
The auther would like to thank Professor Umeda Toru for suggesting this topic with one example in \cite{U}. 


\end{document}